\documentclass[11pt]{article}
\usepackage{amsmath}
\usepackage{amssymb}
\usepackage{amsthm}
\usepackage{bm}
\usepackage{url}
\usepackage{xcolor}
\usepackage{graphicx}

\newcommand\sLP{\\[\smallskipamount]}

\newcommand\bLP{\\[\bigskipamount]}
\newcommand\bPP{\\[\bigskipamount]\indent}

\newcommand\CC{\mathbb{C}}
\newcommand\RR{\mathbb{R}}
\newcommand\ZZ{\mathbb{Z}}
\newcommand\FSH{\mathcal{H}}
\newcommand\FSP{\mathcal{P}}
\newcommand\FST{\mathcal{T}}
\newcommand\al\alpha
\newcommand\be\beta
\newcommand\ga\gamma
\newcommand\de\delta
\newcommand\la\lambda
\newcommand\si\sigma
\newcommand\La\Lambda
\newcommand\iy\infty
\newcommand\wb\overline
\newcommand\wh\widehat
\newcommand\wt\widetilde
\newcommand\union\cup
\newcommand\Lan{\Lambda_{n}}

\newcounter{alphc} 
\newcommand\alphlist{\begin{list}{\bf\alph{alphc})}{\usecounter{alphc} 
 \parsep0.2cm \itemsep0.2cm \topsep0.2cm}\setlength\itemindent{-13pt}} 

\DeclareMathOperator{\sgn}{sgn}
\DeclareMathOperator{\id}{id}

\numberwithin{equation}{section}
\newtheorem{theoremA}{Theorem}
\newtheorem{theorem}{Theorem}[section]
\newtheorem{proposition}[theorem]{Proposition}
\newtheorem{lemma}[theorem]{Lemma}
\newtheorem{corollary}[theorem]{Corollary}
\newtheorem{Definition}[theorem]{Definition}
\newenvironment{definition}{\begin{Definition}\rm}{\end{Definition}}
\newtheorem{Remark}[theorem]{Remark}
\newenvironment{remark}{\begin{Remark}\rm}{\end{Remark}}
\newtheorem{Example}[theorem]{Example}
\newenvironment{example}{\begin{Example}\rm}{\end{Example}}
\newcommand\Proof{\noindent{\bf Proof.}\quad}

\begin{document}
\title{A nonsymmetric version of Okounkov's 
$BC$-type\\
interpolation Macdonald polynomials}
\author{Niels Disveld, Tom H. Koornwinder and Jasper V. Stokman}
\AtEndDocument{\bigskip{\footnotesize%
\textsc{N. Disveld, Eikenweg 48, 1092CA Amsterdam, The Netherlands.}\par
\textit{email address:} \url{nielsdisveld@gmail.com} \par
\addvspace{\medskipamount}
\textsc{T.H. Koornwinder, KdV Institute for Mathematics,
University of Amsterdam,}\par
\textsc{Science Park 105-107, 1098 XG Amsterdam, The Netherlands.}\par
\textit{email address:} \url{thkmath@xs4all.nl} \par
\addvspace{\medskipamount}
\textsc{J.V. Stokman, KdV Institute for Mathematics,
University of Amsterdam,}\par
\textsc{Science Park 105-107, 1098 XG Amsterdam, The Netherlands.}\par
\textit{email address:} \url{j.v.stokman@uva.nl} 
}}
\date{}
\maketitle
\begin{abstract}
\noindent
Nonsymmetric interpolation Laurent polynomials in $n$ variables are
introduced, with the interpolation points depending on $q$ and on a
$n$-tuple of parameters $\tau=(\tau_1,\ldots,\tau_n)$. When $\tau_i=st^{n-i}$ 
Okounkov's $3$-parameter $BC_n$-type interpolation Macdonald polynomials are
recovered from the nonsymmetric interpolation Laurent polynomials
through Hecke algebra symmetrisation with respect to a type $C_n$ Hecke algebra
action.
In the appendix we give some conjectures about extra vanishing, based on
Mathematica computations in rank two.
\end{abstract}
\section{Introduction}

The goal of this paper is to introduce and solve a special class of Newton type
interpolation problems for Laurent polynomials in several variables. 
An important special case leads to nonsymmetric analogs of Okounkov's
\cite{O98,O98p2} $BC_n$-type interpolation Macdonald polynomials.

\subsection{Background: (non)symmetric interpolation polynomials}
Denote by $\mathbb{C}[y]^{S_n}$ the ring of symmetric polynomials in $n$
variables $y=(y_1,\ldots,y_n)$. Let $\Lambda_n^+$ be the set of partitions of length at most
$n$.
Associate to a function 
\[
\Omega^+: \{1,\ldots,n\}\times\mathbb{Z}_{\geq 0}\rightarrow\CC
\]
(called a grid) the interpolation points
$\Upsilon^+(\lambda):=(\Omega^+(1,\la_1),\ldots,\Omega^+(n,\la_n))\in\CC^n$
($\la\in\La_n^+$).
Okounkov \cite{O98p2} showed that under explicit generic conditions on the grid
function $\Omega^+$ there exists, for each $\la\in\La_n^+$, a
$I_\la^+(y)\in\CC[y]^{S_n}$ of total degree at most $|\la|:=\la_1+\cdots+\la_n$
satisfying 
\[
I_\la^+(\Upsilon^+(\mu))=0
\] 
for all $\mu\in\Lambda_n^+\setminus\{\lambda\}$ with $|\mu|\leq |\lambda|$ and
satisfying $I_\la^+(\Upsilon^+(\la))\not=0$. Imposing the normalization
condition $I_\la^+(\Upsilon^+(\la))=1$ fixes $I_\la^+(y)$ uniquely (often a
different normalization is used, but this one is the most convenient choice in
the present paper). 

Okounkov \cite[Def. 4.4]{O98p2} calls a grid perfect when the extra vanishing property
$I_\la^+(\Upsilon^+(\mu))=0$ holds true for all $\la, \mu\in\La_n^+$ such that
$\la\nsubseteq\mu$, where $\subseteq$ is the natural inclusion order on
$\Lambda_n^+$. 
For generic $q,s,t\in\CC^*$ the grid
\begin{equation}\label{BCgrid}
\Omega^+(i,m):=st^{n-i}q^m+\frac{1}{st^{n-i}q^m}
\end{equation}
is perfect, and all other perfect grids can be obtained from \eqref{BCgrid} 
by replacing the grid values $\Omega^+(i,m)$ by $a\Omega^+(i,m)+b$ where
$a\in\mathbb{C}^*$ and $b\in\mathbb{C}$ are independent of $i$ and $m$
(but they may depend on the parameters) and by taking limits
(see \cite{O98p2}).

The interpolation polynomials $I_\la^+(y)$ ($\la\in\La_n^+$) associated to the
grid \eqref{BCgrid} are Okounkov's \cite{O98} $3$-parameter family of
$BC_n$-type interpolation Macdonald polynomials.
They admit explicit $q$-integral and combinatorial representations, see \cite{O98,O03},
while a generalized binomial formula provides an explicit expansion of the
Koornwinder polynomial in the $I_\la^+(y)$'s, see \cite{O98,O03}.
Rains \cite[Thm. 3.2]{R05} characterized the $I_\la^+(y)$'s as eigenfunctions
of a linear $q$-difference operator acting on $y_1,\ldots,y_n$ as well as the
parameter $s$. 

An important degenerate perfect grid is
\begin{equation}\label{Agrid}
\Omega^+(i,m):=t^{n-i}q^m.
\end{equation}
In this case the corresponding interpolation polynomials $I_\la^+(y)$
($\la\in\La_n^+$) are the interpolation Macdonald polynomials of
Knop and Sahi \cite{S96,Kn97}, having the Macdonald polynomials as their top
homogeneous components. Their fundamental properties ($q$-difference equations,
extra vanishing, $q$-integral formula, combinatorial formula, binomial formula)
were obtained at the end of the past century \cite{Kn97,O97,O98p0,S96}. 

The theory of symmetric interpolation polynomials associated to perfect grids originates from the study of the Capelli identity and its generalisations, see \cite{KS,S94}.  
By now various classes of symmetric interpolation polynomials $I_\lambda^+(y)$ associated to perfect grids have been realised as eigenvalues of Capelli operators or as images under the Harish-Chandra isomorphism of quantum immanants (see, e.g., \cite{S94,O96,O03,Na,MN,SaSa} and references therein). They have found applications in, e.g., the theory of multivariable special functions associated to classical root systems (see, e.g., \cite{O98,R05,R10}), exactly solvable models (see, e.g., \cite{Na,O03,S96,BWZ}), and infinite dimensional harmonic analysis on Lie groups and symmetric spaces \cite{OO}. An important recent development is the generalisation of the $\textup{BC}_n$-type interpolation Macdonald polynomials to the elliptic level, see \cite{R06, CG}.

Macdonald and Koornwinder polynomials have natural nonsymmetric counterparts,
see \cite{Ch95, No, Sahi99}. 
They are the joint polynomial eigenfunctions of Cherednik's commuting
$q$-difference reflection operators, which in turn constitute part of
Cherednik's \cite{Chbook} polynomial representation of the double affine Hecke
algebra in terms of Demazure--Lusztig operators. 

Nonsymmetric counterparts $I_\al(y)\in\CC[y]$ ($\al\in\ZZ_{\geq 0}^n$) of the
interpolation polynomials $I_\la^+(y)$ were 
introduced in \cite{S96,Kn97} for grids of the form
\begin{equation}\label{Ageneralgrid}
\Omega(i,m):=\tau_iq^m.
\end{equation} 
The associated interpolation points are
$\Upsilon(\be):=\bm{u_\be}(\Upsilon^+(\be^+))\in(\CC^*)^n$ 
($\be\in\ZZ_{\geq 0}^n$), where $\be^+\in\La_n^+$ denotes the unique partition
in the $S_n$-orbit of $\be$ and $u_\be\in S_n$ is the element of
smallest length such that $u_\be(\be^+)=\be$ (here we follow the notations from
Section \ref{sec:2} for the permutation actions of $S_n$ on $(\CC^*)^n$ and
$\ZZ^n$). Up to normalization, $I_\al(y)\in\CC[y]$ is characterized as the
nonzero polynomial of degree at most $|\al|$ such that $I_\al(\Upsilon(\be))=0$
for all $\be\in\ZZ_{\geq 0}^n\setminus\{\al\}$ satisfying $|\be|\leq |\al|$.

Note that in the principal specialization $\tau_i:=t^{n-i}$, the grid
\eqref{Ageneralgrid} reduces to the grid \eqref{Agrid}. In this case the
$I_\al(y)$ ($\al\in\ZZ_{\geq 0}^n$) are nonsymmetric analogs of the
interpolation Macdonald polynomials whose top degree components are
nonsymmetric Macdonald polynomials. They satisfy extra vanishing conditions,
have a natural duality property, are simultaneous eigenfunctions on
inhomogeneous versions of Cherednik operators, and they admit explicit binomial
and evaluation formulas, see \cite{Kn97,S98,SaSt}. 
{\let\thefootnote\relax\footnote{{2010 {\it Mathematics
Subject Classification.} Primary 33D52, 05E05;
Secondary 33D80, 05E10.}}}

\subsection{Nonsymmetric interpolation Laurent polynomials}

In this paper we introduce nonsymmetric counterparts of the interpolation
polynomials $I_\la^+(y)$ for the grid
\begin{equation}\label{BCgeneralgrid}
\Omega^+(i,m):=\tau_iq^m+\frac{1}{\tau_iq^m}.
\end{equation}
In the principal specialization $\tau_i:=st^{n-i}$, the grid
\eqref{BCgeneralgrid} reduces to the perfect grid \eqref{BCgrid} underlying the
$BC_n$-type interpolation Macdonald polynomials.

In case of the grid \eqref{BCgeneralgrid} it is instrumental to look for
nonsymmetric analogs of the symmetric interpolations within the space
$\CC[x^{\pm 1}]$ of {\it Laurent} polynomial in $n$ variables $x=(x_1,\ldots,x_n)$. 
Consider $\CC[y]$ as the subalgebra of $\CC[x^{\pm 1}]$ generated by
$y_i:=x_i+x_i^{-1}$. Then $\CC[y]^{S_n}$ is the algebra $\CC[x^{\pm 1}]^{W_n}$
of $W_n$-invariant Laurent polynomials, where $W_n$ is
the type $C_n$ Weyl group acting by permutations and inversions of the
variables. Note that the interpolation points for $I_\la^+(y)$ in terms of the
$x$-variables are given by 
\[
\Upsilon(\mu):=(\Omega(1,\mu_1),\ldots,\Omega(n,\mu_n))\qquad (\mu\in\La_n^+), 
\]
with $\Omega$ the type $A_{n-1}$ grid defined by \eqref{Ageneralgrid}.

The Weyl group $W_n$ acts on $(\CC^*)^n$ by permutations and inversions,
and on $\ZZ^n$ by permutations and sign changes.
Write $\be^+\in\La_n^+$ for $\be\in\ZZ^n$ the unique partition in the
$W_n$-orbit of $\beta$, and $w_\be\in W_n$ for the element of smallest length
such that $w_\be(\be^+)=\be$. We call $|\al|:=|\al_1|+\cdots+|\al_n|$
the degree of the monomial $x^\al\in\CC[x^{\pm 1}]$ ($\al\in\ZZ^n$). 

Our main result is as follows (see Section \ref{sec:4}).

\begin{theoremA} 
Assume $q\in\CC^*$ is not a root of unity, and denote by $\FST_n$ the generic
set of parameters $\tau=(\tau_1,\ldots,\tau_n)$ defined by \eqref{eq:9}. 
For $\tau\in\FST_n$ and $\al\in\ZZ^n$ there exists a unique Laurent polynomial
$G_\al(x)=G_\al(x;q,\tau)\in\CC[x^{\pm 1}]$ of degree at most $|\al|$
satisfying $G_\al(\overline{\al})=1$ and 
\[
G_\al(\overline{\be})=0\qquad (\be\in\ZZ^n\setminus\{\al\}:\,\,
|\be|\leq |\al|),
\]
where $\overline{\be}:={\bm w_\be}(\Upsilon(\be^+))$. 
Furthermore, 
\[
I_\la^+(y)=\sum_{\al\in W_n\la}G_\al(x)\qquad (\la\in\La_n^+)
\] 
with $I_\la^+(y)$ the symmetric interpolation polynomial relative to the
grid \eqref{BCgeneralgrid}.
\end{theoremA}
The proof of the existence of the interpolation polynomials $I_\la^+(y)$ and $I_\al(y)$ are based on explicit recursion relations, which allow a direct proof by induction to the degree
(see \cite[Prop. 2.7]{O98p2} and \cite[Cor. 4.4]{S96}). We revisit the proof for $I_\la^+(y)$ with the grid \eqref{BCgeneralgrid} from the Laurent polynomial perspective in Section \ref{sec:3}. It forms a convenient starting point for the much more elaborate inductive proof of Theorem 1, which is given in Section \ref{sec:4}. 

The nonsymmetric interpolation polynomials $G_\al(x;q,s,t)$ ($\al\in\ZZ^n$) with parameters $\tau=(\tau_1,\ldots,\tau_n)$ specialized to the $\tau_i:=st^{n-i}$ are nonsymmetric analogs of Okounkov's $BC_n$-type interpolation Macdonald polynomials. In this case the $BC_n$-type interpolation Macdonald polynomials can alternatively be reobtained from the $G_\al(x;q,s,t)$'s by symmetrising with respect to a type $C_n$ Hecke algebra action on $\CC[x^{\pm 1}]$ in terms of Demazure--Lusztig type operators (see Section \ref{sec:5}). It is a first indication that the $G_\al(x;q,s,t)$ are amenable to the Hecke algebra techniques from \cite{Kn97,S98,SaSt}. 
The missing ingredient from this perspective is the interpretation of the $G_\al(x;q,s,t)$'s as simultaneous eigenfunctions of commuting inhomogeneous Cherednik-type operators. This would allow one to involve double affine Hecke algebra techniques in deriving a binomial formula for nonsymmetric Koornwinder polynomials, and in deriving nonsymmetric analogs of extra vanishing and duality (compare with \cite{Kn97,SaSt} for type $A$). We expect this to be the key step towards further applications of the nonsymmetric $\textup{BC}_n$-type interpolation polynomials in the theory of non-symmetric Macdonald-Koornwinder polynomials, algebraic combinatorics and exactly solvable models.

In the appendix we give a conjecture about the extra vanishing, based on Mathematica computations in rank
two.\bLP
{\bf Acknowledgments.}
The authors thank Siddhartha Sahi and Eric Rains for valuable
discussions and comments.
The second and third author thank Masatoshi Noumi for sharing with us
his insight on $BC_n$-symmetric interpolation polynomials.
A substantial
part of
sections 3 and 4 is based on
material in the Master's
Thesis by the first author under supervision of the last two authors (University of Amsterdam, Faculty of
Science, 2017). We thank the referees for valuable comments that led to significant improvements of the text.


\section{Preliminaries}
\label{sec:2}

Throughout the paper we assume that 
$q\in\CC^*:=\CC\setminus\{0\}$ is not a root of unity.
For $a\in\CC$ the {\em $q$-shifted factorial} is given by
$(a;q)_k:=(1-a)(1-qa)\ldots(1-q^{k-1}a)$ ($k=1,2,\ldots$) and
$(a;q)_0:=1$. We also write
$(a_1,\ldots,a_r;q)_k:=(a_1;q)_k\ldots(a_r;q)_k$. 

Let $n\in\ZZ_{>0}$. Write
\[
\Lan:=\ZZ^n,\qquad [1,n]:=\{1,\ldots,n\},\qquad [1,n)=\{1,\ldots,n-1\},
\] 
where $[1,n)$ is taken to be empty if $n=1$. 

For $x=(x_1,\ldots,x_n)\in\CC^n$ and $\al=(\al_1,\ldots,\al_n)\in\Lan$ put
\begin{align*}
|\al|&:=|\al_1|+\cdots+|\al_n|\quad\;\;(\mbox{\emph{weight} of
$\al$}),\\
n_a(\al)&:=\#\{i\,\,\,\vert\,\,\, \al_i=a\}\qquad\quad\mbox{if $a\in\ZZ$},\\
x^\al&:=x_1^{\al_1}\ldots x_n^{\al_n}\qquad\qquad
(\mbox{Laurent monomial}),\\
x_I&:=\prod_{i\in I}x_i\qquad\qquad\qquad\;\mbox{if $I\subseteq[1,n]$}.
\end{align*}
When we write $I=\{i_1,\ldots,i_k\}\subseteq [1,n]$ for a subset of
$[1,n]$ of cardinality $k$, then we will always assume the ordering
$i_1<i_2<\cdots<i_k$ of its elements.  We write $I^\mathsf{c}$ for the
complement of $I$ in $[1,n]$.

A {\em Laurent polynomial} $f$ in the $n$ complex
variables $x$ has the form
\begin{equation}
f(x)=\sum_{\al\in\Lan} c_\al x^\al
\label{eq:1}
\end{equation}
with $c_\al\in\CC$ and $c_\al\ne0$ for only finitely many $\al$.
The {\em degree} of $f$ in \eqref{eq:1} is defined by
\begin{equation*}
\deg(f):=\max\big\{|\al|\;\big|\;c_\al\ne0\big\};\qquad
\deg(f):=-\iy\quad\mbox{if $f$ is identically zero}.
\end{equation*}
Note that $\deg(fg)\le\deg(f)+\deg(g)$, but equality does not
necessarily hold. 
This means that the degree function defines a
filtration on the algebra
$\FSP_n:=\CC[x_1^{\pm 1},\ldots,x_n^{\pm 1}]$ of Laurent polynomials in $x$,
but not a grading. The
filtration is $\FSP_{n}=\bigcup_{d=0}^{\iy}\FSP_{n,d}$ with
$\FSP_{n,d}$ the subspace of $\FSP_n$ consisting of
Laurent polynomials of degree at most $d$. 
Write
$\mathcal{G}\bigl(\FSP_n\bigr)=
\bigoplus_{d=0}^{\iy}\mathcal{G}_{n,d}$
for the associated graded algebra, with
the $d$th graded piece given by
\[
\mathcal{G}_{n,d}:=\FSP_{n,d}/\FSP_{n,d-1}
\]
with $\FSP_{n,-1}:=\{0\}$. We write $[f]_d:=
f+\FSP_{n,d-1}\in\mathcal{G}_{n,d}$
for $f\in\FSP_{n,d}$.

Let $W_n=\{\pm1\}^n\rtimes S_n$ be the Weyl group associated with the root
system of type $C_n$ (and $B_n$ and $BC_n$). Then $\si\in\{\pm1\}^n$ and
$\pi\in S_n$ act on  $\Lan$ (and $\CC^n$) by
\begin{equation}
(\si\al)_j:=\si_j\al_j,\quad
(\pi\al)_j:=\al_{\pi^{-1}(j)}\qquad
(\al\in\Lan\mbox{ or }\CC^n,\;j=1,\ldots,n).
\label{eq:20}
\end{equation}
Equivalently, if $e_1,\ldots,e_n$ is the standard basis of $\RR^n$ then
$\si(e_j)=\si_j\,e_j$ and $\pi(e_j)=e_{\pi(j)}$.
Since $2\pi i\Lan$ is invariant under the action of $W_n$, we can
exponentiate its action on $\CC^n$ to an action on $(\CC^*)^n$.
We will write this action in bold. Then
\begin{equation}
(\bm{\si}x)_j:=x_j^{\si_j},\quad
(\bm{\pi}x)_j:=x_{\pi^{-1}(j)}\qquad
(x\in(\CC^*)^n,\;j=1,\ldots,n)
\label{eq:21}
\end{equation}
and $(\bm{w}x)^\al=x^{w^{-1}\al}$ ($x\in(\CC^*)^n$, $w\in W_n$, 
$\al\in\Lan$).
The action of $W_n$ on  $(\CC^*)^n$ induces
an action of $W_n$ on Laurent
polynomials \eqref{eq:1} by
\begin{equation}
(wf)(x):=f(\bm{w^{-1}}x)\qquad(f\in\FSP_n,\;w\in W_n,\;x\in (\CC^*)^n).
\label{eq:22}
\end{equation}
Thus $(wf)(x):=x^{w\al}$ if $f(x)=x^\al$. Note that $|w\al|=|\al|$ for
$w\in W_n$ and $\al\in\Lan$. In particular,
$\textup{deg}(wf)=\textup{deg}(f)$ for $w\in W_n$ and $f\in
\FSP_n$. 
Hence the $W_n$-action on $\FSP_n$ is an action by filtered algebra automorphisms and induces a $W_n$-action by graded
algebra automorphisms on the associated graded algebra
$\mathcal{G}(\FSP_n)$. We write $\FSP_n^{W_n}$, $\FSP_{n,d}^{W_n}$ and
$\mathcal{G}_{n,d}^{W_n}$ for the subspaces of $W_n$-invariant
elements in $\FSP_n$, $\FSP_{n,d}$ and $\mathcal{G}_{n,d}$,
respectively. By construction the associated graded algebra
$\mathcal{G}(\FSP_n^{W_n})$ of the filtered algebra
$\FSP_n^{W_n}=\bigcup_{d=0}^{\iy}\FSP_{n,d}^{W_n}$ is isomorphic to
$\mathcal{G}(\FSP_n)^{W_n}=\bigoplus_{d=0}^{\iy}
\mathcal{G}_{n,d}^{W_n}$.

If $n>1$ then we write 
\[
x^\prime:=(x_1,\ldots,x_{n-1})
\]
for the $n-1$ complex variables obtained from $x$ by removing
$x_n$. Similarly, if $\tau=(\tau_1,\ldots,\tau_n)$ is a $n$-tuple of
complex numbers, then we write
\[
\tau^\prime:=(\tau_1,\ldots,\tau_{n-1}).
\]
Sometimes we also need to remove an arbitrary complex variable $x_k$
from $x$. In that case we write
\[
x^{(k)}:=
(x_1,\ldots,x_{k-1},x_{k+1},\ldots,x_n).
\]
A similar notation will be employed for $n$-tuples of complex
numbers. Note that $x^\prime=x^{(n)}$ and $\tau^\prime=\tau^{(n)}$.

For the root system $R$ of $C_n$ we take
$\be^i:=e_i-e_{i+1}$ ($i=1,\ldots,n-1$) and
$\be^n:=2e_n$ as the simple roots.
Then the set $R^+$ of positive roots consists of
 the vectors $e_i\pm e_j$ ($i<j$) and $2e_i$ ($i=1,\ldots,n$).
 Write $R^-=-R^+$ for the set of negative roots.
Denote the  simple reflections corresponding to the
simple roots by $s_1,\ldots,s_n$.
Each $w\in W_n$ can be written as a product of simple reflections.
The minimal number of factors in such a product representing $w$ is
called the {\em length} $\ell(w)$ of $w$.
The length of $w$ is also equal to the
number of positive roots sent to negative roots by $w$, see
\cite[Lemma 10.3A]{Hum79} or \cite[Corollary 1.7]{Hum90}.
If $w=s_{i_1}\ldots s_{i_r}$ with $r=\ell(w)$ (a {\em reduced expression}
of $w\in W_n$) then $w_k:=s_{i_k}\ldots s_{i_r}$ gives a reduced expression
for $w_k$ ($k=1,\ldots,r$) and the $r-k+1$ positive roots sent by $w_k$
to negative roots are precisely the positive roots
$s_{i_r} s_{i_{r-1}}\ldots s_{i_{j+1}}\be^{i_j}$ ($j=k,\ldots,r$),
see \cite[Section 1.7]{Hum90}.
The element $w_0$ of maximal length
in $W_n$ is $w_0=-\id_{\RR^n}$.

By a {\em partition} we mean
$\la\in\Lan$ with $\la_1\ge\la_2\ge\cdots\ge\la_n\ge0$.
The \emph{length} $\ell(\la)\in\{0,1,\ldots,n\}$ of $\la$  is the index
such that $\la_i=0$ iff $i>\ell(\la)$.
Denote the set of partitions of length at most $n$
by~$\La_n^+$. For $\al\in\Lan$ there exists a unique partition
$\al^+\in\La^+_n$ in the $W_n$-orbit $\{w\al\}_{w\in W_n}$. 
If $\al\in\ZZ_{\geq 0}^n$ then $\al^+$ can also be characterized as
the unique partition in the $S_n$-orbit $\{\pi\al\}_{\pi\in S_n}$.
For $m\le n$ we embed $\La_m^+\hookrightarrow \La_n^+$ by
$(\la_1,\ldots,\la_m)\mapsto(\la_1,\ldots,\la_m,0,\ldots,0)$.

On $\La_n^+$ the \emph{dominance partial ordering} $\le$
and \emph{inclusion partial ordering} $\subseteq$ are defined by
\begin{align}
\la\le\mu\quad&\mbox{iff}\quad \la_1+\cdots+\la_i\le\mu_1+\cdots+\mu_i\quad
(i=1,\ldots,n),
\label{eq:5}\\
\la\subseteq\mu\quad&\mbox{iff}\quad \la_i\le\mu_i\quad
(i=1,\ldots,n).\nonumber
\end{align}
A partition $\la$ has weight $|\la|=\la_1+\cdots+\la_n$.
If $\la\le\mu$ then $|\la|\le|\mu|$.

For $\la\in\La_n^+$ write
$W_{n,\la}\subseteq W_n$ for the stabilizer subgroup of $\la$.
Then by \cite[Lemma 10.3B]{Hum79}, the subgroup
$W_{n,\la}$ ($\la\in\La_n^+$)
is a {\em parabolic subgroup} of $W_n$, i.e., $W_{n,\la}$ is 
generated by the simple reflections it contains.
Write $W_n^\la$ for the set of $u\in W_{n}$ satisfying $\ell(uv)=\ell(u)+\ell(v)$ for all $v\in W_{n,\la}$.
Then $W_n^\la$ is a complete set of representatives of $W_n/W_{n,\la}$
(see \cite[Section 1.10]{Hum90} for details).

\begin{definition}
\label{def:2.1}
For $\al\in\Lan$ denote by $w_\alpha=\si_\al\pi_\al$ ($\si_\al\in\{\pm 1\}^n$, $\pi_\al\in S_n$) the unique element in $W_n^{\alpha^+}$ 
such that $w_\al(\al^+)=\al$.
\end{definition}

Define the function $\sgn\colon\RR\to\{\pm1\}$ by
$\sgn(a):=1$ if $a\ge0$ and $\sgn(a):=-1$ if $a<0$.
The following explicit description
of $\si_\al$ and $\pi_\al$,
due to Sahi \cite[p.~277]{Sahi99}, plays an important role in what follows.

\begin{lemma}
\label{th:2}
Let $\al\in\Lan$. Then 
\[
\si_\al=(\sgn(\al_1),\ldots,\sgn(\al_n)),
\] 
and
$\pi_\al\in S_n$ is the unique permutation satisfying the following two properties:
\begin{enumerate}
\item[{\bf a.}] $\al^+=(|\al_{\pi_\al(1)}|,\ldots,|\al_{\pi_\al(n)}|)$;
\item[{\bf b.}] Let $a\in\ZZ_{\geq 0}$ such that $n_a(\al^+)>0$. Denote by 
$j_1<j_2$ the indices such that
$\{j\mid\al_j^+=a\}=\{j_1,j_1+1,\ldots,j_2-1\}$. 
\begin{enumerate}
\item If $j_1\le j<j_1+n_a(\al)$ then $\al_{\pi_\al(j)}=a\ge0$ and $\pi_\al(j)$
is increasing in $j$,
\item if $j_1+n_a(\al)\le j<j_2$ then $\al_{\pi_\al(j)}=-a<0$ and $\pi_\al(j)$
is decreasing in $j$.
\end{enumerate}
\end{enumerate}
\end{lemma}

\begin{remark}
Alternatively, $\pi_\al\in S_n$ is the unique permutation such that, for $i<j$,
\begin{equation}\label{altdef}
\pi_\al^{-1}(i)<\pi_\al^{-1}(j)\,\, \Leftrightarrow
\,\,|\al_i|>|\al_j| \hbox{ or }
0\le\al_i=\pm\al_j.
\end{equation}
\end{remark}

See also \cite[p.~277]{Sahi99} for a short formulation of what is essentially
the description of $\pi_\al$ in Lemma \ref{th:2}, and for an example.
\section{Interpolation theorem for $\textup{BC}_n$-symmetric Laurent polynomials}
\label{sec:3}
Following \cite{Ra}, we call a Laurent polynomial \eqref{eq:1} \emph{$BC_n$-symmetric} if it
is invariant under the Weyl group $W_n$. A basis of the linear space
of $\textup{BC}_n$-symmetric Laurent polynomials is given by the
\emph{symmetrized monomials}
\begin{equation*}
m_\la(x):=\sum_{\mu\in W_n\la} x^\mu\qquad(\la\in\La_n^+).
\end{equation*}
The Laurent polynomial $m_\la$ has degree $|\la|$, and 
\begin{equation*}
m_\la=\frac{|W_n^\la|}{|W_n|}\, \tilde m_\la,\quad\mbox{where}\quad
\tilde m_\la(x):=\sum_{w\in W_n} x^{w\la}.
\end{equation*}
\begin{lemma}
\label{th:4}
Let $a\in\CC\backslash\{0\}$. Let $\la\in\La_{n-1}^+\hookrightarrow\La_n^+$
($n>1$).
\alphlist
\item
There are constants $c_\mu$ ($\mu\in\La_{n-1}^+$, $\mu\le\la$)
with $c_\la\ne0$ such that
\begin{equation}
m_\la(x',a)=\sum_{\mu\le\la}c_\mu m_\mu(x')
\label{eq:7}
\end{equation}
as identity in $\FSP_{n-1}^{W_{n-1}}$.
\item
There are constants $d_\mu$ ($\mu\in\La_{n-1}^+$, $\mu\le\la$)
with $d_\la\ne0$ such that
\begin{equation}
m_\la(x')=\sum_{\mu\le\la}d_\mu m_\mu(x',a)
\label{eq:6}
\end{equation}
as identity in $\FSP_{n-1}^{W_{n-1}}$.
\item
For every $\textup{BC}_{n-1}$-symmetric Laurent polynomial $f$ in $x^\prime$
of degree $d$ there exists a $\textup{BC}_n$-symmetric Laurent polynomial $g$ in $x$
of degree $d$ such that $g(x^\prime,a)=f(x^\prime)$.
\end{list}
\end{lemma}
\Proof 
{\bf a)} Denote the length of $\la\in\La_{n-1}^+$ by $\ell$. Then
\begin{equation*}
\tilde m_\la(x',a)=\sum_{j=1}^n(a^{\la_j}+a^{-\la_j})\,\tilde m_{\la^{(j)}}(x')
=2(n-\ell)\tilde m_\la(x')+\sum_{j=1}^\ell(a^{\la_j}+
a^{-\la_j})\,\tilde m_{\la^{(j)}}(x').
\end{equation*}
The result now follows from the fact that $\la^{(j)}<\la$ for $j=1,2,\ldots,\ell$.\\
{\bf b)}
From \eqref{eq:7} we get
\[
m_\la(x')=c_\la^{-1} m_\la(x',a)-\sum_{\mu<\la}c_\la^{-1}c_\mu m_\mu(x').
\]
Now \eqref{eq:6} follows by induction on the weight of the partition.\\
{\bf c)}
By \eqref{eq:6}, if $f=m_\la\in\FSP_{n-1}^{W_{n-1}}$ then we can take
$g=\sum_{\mu\in\La_{n-1}^+;\;\mu\le\la}d_\mu m_\mu\in\FSP_n^{W_n}$.\qed
\bPP

Write $R_A:=\{e_i-e_j\mid 1\leq i\not=j\leq n\}\subset R$, which is a root subsystem in $R$ of type $A_{n-1}$.
Then we define the parameter domain $\FST_n$ by
\begin{equation}
\FST_n:=\{\tau\in (\CC^*)^n\mid \tau^\be\not\in q^{\ZZ}\quad \forall\, \beta\in R\setminus R_A\quad \&\quad \tau^\be\not\in q^{\ZZ\setminus\{0\}}
\quad \forall\, \beta\in R_A\},
\label{eq:9}
\end{equation}
where $\tau^\al\in\CC^*$ ($\al\in\La_n$) stands for the monomial $x^\al$ evaluated at $\tau\in(\CC^*)^n$. In other words, $\tau\in (\CC^*)^n$ belongs to $\FST_n$ if $\tau_i\tau_j\not\in q^{\ZZ}$
for $1\leq i\leq j\leq n$ and $\tau_i\tau_j^{-1}\not\in q^{\ZZ\setminus\{0\}}$ for $1\leq i<j\leq n$. Note that
$\mathbf{s}:=(s,\ldots,s)\in\FST_n$ for $s\in\CC^*$ satisfying $s^2\not\in q^\ZZ$, and $\FST_n$ is invariant under scalar multiplication by 
$q$. For $\mu\in\La_n^+$ define $\wb\mu=(\wb\mu_1,\ldots,\wb\mu_n)\in(\CC^*)^n$
by
\begin{equation}
\wb\mu_i:=q^{\mu_i}\tau_i.
\label{eq:8}
\end{equation}
The map $\mu\mapsto\wb\mu$ is injective on $\La_n^+$. In particular,
Sometimes we write $\overline{\mu}=\overline{\mu}(q,\tau)$
and $\overline{\mu}_i=\overline{\mu}_i(q,\tau)$
if it is important to specify the dependence on $q,\tau$.

\begin{remark}\label{parameterremark}
We will develop the
 theory of symmetric and nonsymmetric 
 interpolation Laurent polynomials for parameters $(q,\tau)$ with
$q\in\CC^*$ not a root of unity and $\tau\in\FST_n$ (see \eqref{eq:9}). It is easy to check that the results also hold true 
with $q$ and $\tau_j$ rational indeterminates, 
and for $(q,\tau)$ with $0<|q|<1$ and $0<|\tau_1|<\cdots<|\tau_n|<1$. The latter case requires straightforward adjustments to the proofs of Corollary \ref{cor:9} and Lemma \ref{lem:rf}.
\end{remark}

The following two properties of the interpolation points will play an
important role in what follows,
\begin{equation}\label{special}
\begin{split}
\overline{\mu}(q,\tau)&=(\overline{\mu}(q,\tau^\prime),\tau_n)
\qquad\,\, \,\,\,\,\,\,\quad\,\,
(\mu\in\La_{n-1}^+\hookrightarrow\La_n^+),\\
\overline{\mu}(q,\tau)&=\overline{(\mu-\mathbf{1})}(q,q\tau)\qquad
\qquad(\mu\in\La_{n}^+:\,\, \mu_n>0),
\end{split}
\end{equation}
with $\overline{\mu}(q,\tau^\prime)=
(q^{\mu_1}\tau_1,\ldots,q^{\mu_{n-1}}\tau_{n-1})$ for
$\mu\in\La_{n-1}^+$ the interpolation point
in $(\CC^*)^{n-1}$. 

Put
\begin{equation*}
\La_{n,d}^+:=\{\mu\in\La_n^+\mid |\mu|\le d\}\qquad (d\in\ZZ_{\geq 0}).
\end{equation*}

\begin{proposition}
\label{th:3}
Let $n\in\ZZ_{>0}$, $d\in\ZZ_{\ge0}$ and $\tau\in\FST_n$.
For every map $\wb f\colon \La_{n,d}^+\to\CC$
there exists a unique
$\textup{BC}_n$-symmetric Laurent polynomial $f$ 
of degree
$\le d$ such that
$f(\wb \mu(q,\tau))=\wb f(\mu)$ for all $\mu\in\La_{n,d}^+$.
\end{proposition}
\Proof
First note that both the space $\FSP_{n,d}^{W_n}$ of $\textup{BC}_n$-symmetric Laurent
polynomials
in $x$ of degree at most $d$
and the space of complex-valued functions
on $\{\wb\mu\mid\mu\in\La_{n,d}^+\}$ have dimension $|\La_{n,d}^+|$.
Therefore, surjectivity of the linear map which restricts a $\textup{BC}_n$-symmetric Laurent
polynomial to the set of interpolation points
$\{\wb\mu\mid\mu\in\La_{n,d}^+\}$ implies injectivity, so that existence implies uniqueness.

To prove existence we will use induction on $n+d$.
If $n+d=1$, so $(n,d)=(1,0)$, then $\La_{1,0}^+=\{(0)\}$ and
$\overline{0}=\tau_1$ and there is nothing to prove (take $f$ to be
the appropriate constant function).  Suppose that the existence of the
symmetric interpolation Laurent polynomial, with $(n,d)$ replaced by
$(\widetilde{n},\widetilde{d})$, is true for $\widetilde{n}+\widetilde{d}<n+d$ for all
possible parameters in $\FST_{\widetilde{n}}$ and all possible maps
$\La_{\widetilde{n},\widetilde{d}}^+\rightarrow\CC$.
Fix $\wb f\colon
\La_{n,d}^+\to\CC$ and $\tau\in\FST_n$ and let $\overline{\mu}$ be
$\overline{\mu}(q,\tau)$. To establish the induction step, we need to
prove the existence of a $f\in\FSP_{n,d}^{W_n}$ satisfying
$f(\overline{\mu})=\overline{f}(\mu)$ for all $\mu\in\La_{n,d}^+$.

We first construct a $g\in\FSP_{n,d}^{W_n}$ satisfying the partial
interpolation property
\begin{equation}\label{substep}
g(\overline{\mu})=\overline{f}(\mu)\qquad
(\mu\in\La_{n,d}^+\mbox{ with }\mu_n=0).
\end{equation}

First assume $n>1$. 
By induction, there exists a $\textup{BC}_{n-1}$-symmetric Laurent polynomial $\widetilde{g}$
in $x^\prime$ of degree at most $d$
such that
\begin{equation}
\widetilde{g}\big(\,\wb{\mu}(q,\tau^\prime)\,\big)=\wb f(\mu,0)\qquad
(\mu\in\La_{n-1,d}^+).
\label{eq:10}
\end{equation}
By Lemma \ref{th:4}c) there exists a $\textup{BC}_n$-symmetric Laurent polynomial $g$ in
$x$ of degree at most $d$
such that
$g(x^\prime,\tau_n)=\widetilde{g}(x^\prime)$. Then
\[
g\big(\,\wb{\mu}(q,\tau'),\tau_n\big)=\wb f(\mu,0)\qquad(\mu\in\La_{n-1,d}^+).
\]
Now the first formula of \eqref{special} gives
\eqref{substep}.
If $n=1$ then put $g(x):=\wb f(0)$, where $g$ has degree $d\ge0$.
Then, in particular, $g(\tau_1)=\wb f(0)$. This concludes the
proof of \eqref{substep} in all cases.

Note that \eqref{substep} already concludes the proof of the induction
step when $n>d$. Indeed, in this case we can simply take $f=g$ since
$\mu_n=0$ for all $\mu\in\La_{n,d}^+$.

To complete the induction step we thus may and will assume from now on
that $d\geq n$.
We make the Ansatz that the symmetric interpolation Laurent polynomial
$f$ we are searching for is of the form
\begin{equation}
f(x)=g(x)+h(x)\prod_{i=1}^n(x_i-\tau_n)(x_i^{-1}-\tau_n)
\label{eq:11}
\end{equation}
with $g$ as constructed above and $h\in\FSP_{n,d-n}^{W_n}$.
Then $f(\wb \mu)=\wb f(\mu)$ for all $\mu\in\La_{n,d}^+$ with $\mu_n=0$.
The identity $f(\wb \mu)=\wb f(\mu)$
will also hold for $\mu_n>0$ if $h$ satisfies
\begin{equation}
h(\wb\mu)=\frac{\wb f(\mu)-g(\wb\mu)}
{\prod_{i=1}^n(\wb\mu_i-\tau_n)(\wb\mu_i^{\,-1}-\tau_n)}
\qquad(\mu\in\La_{n,d}^+\mbox{ with }\mu_n>0).
\label{eq:12}
\end{equation}
Note that, since $\tau\in\FST_n$, no factors in the above denominator
vanish. 
So what remains to show is the existence of a $\textup{BC}_n$-symmetric Laurent
polynomial $h\in\FSP_{n,d-n}^{W_n}$ satisfying \eqref{eq:12}.

Note that we have a bijection
\[
\{\mu\in\La_{n,d}^+\mid \mu_n>0\}\overset{\sim}{\longrightarrow} \La_{n,d-n}^+
\]
given by $\mu\mapsto \mu-\mathbf{1}:=(\mu_1-1,\ldots,\mu_n-1)$. 
By the induction hypothesis, there exists a $h\in\FSP_{n,d-n}^{W_n}$
such that
\[
h(\overline{\nu}(q,q\tau))=\frac{\overline{f}(\nu+\mathbf{1})-
g(\overline{\nu+\mathbf{1}})}
{\prod_{i=1}^n\bigl(\overline{(\nu+\mathbf{1})}_i-
\tau_n\bigr)\bigl(\overline{(\nu+\mathbf{1})}_i^{\,-1}-\tau_n\bigr)}\qquad
(\nu\in\La_{n,d-n}^+).
\]
By the second formula of \eqref{special} we have
\[
\overline{\nu+\mathbf{1}}=\overline{\nu}(q,q\tau)\qquad (\nu\in\La_{n,d-n}^+),
\]
hence we conclude that $h\in\FSP_{n,d-n}^{W_n}$ satisfies the desired
interpolation property \eqref{eq:12}. This concludes the proof of the
induction step.\qed
\bPP
In view of Proposition \ref{th:3}
we can give the following definition.
\begin{definition}
\label{th:5}
Fix $\tau\in\FST_n$. 
The $BC_n$-symmetric interpolation Laurent polynomial of degree $\la\in\La_n^+$ is
the unique $\textup{BC}_n$-symmetric Laurent polynomial $R_\la(x;q,\tau)$ in $n$ variables
$x$ of degree at most $|\la|$
such that $R_\la(\wb\la;q,\tau)=1$ and 
\begin{equation*}
R_\la(\wb\mu;q,\tau)=0\qquad\forall\,\mu\in\La_{n,|\la|}^+\setminus\{\la\},
\end{equation*}
where $\wb\mu:=\wb\mu(q,\tau)$. 
\end{definition}
%
It follows from Proposition \ref{th:3} that
$\{R_\la(x;q,\tau)\mid \la\in\La_{n,d}^+\}$
is a linear basis of $\FSP_{n,d}^{W_n}$.

The following two properties, which correspond to \cite[Prop. 2.4]{O98p2} and
\cite[Prop. 2.7]{O98p2} respectively in Okounkov's setup, easily follow from Definition \ref{th:5},
Proposition \ref{th:3} and the two special properties \eqref{special}
of the interpolation points.
%
\begin{proposition}\label{ElementaryProp}
Let $\tau\in\FST_n$.
\label{th:6}
\alphlist
\item For $\la\in\La_{n-1}^+\hookrightarrow\La_n^+$ we have
\[
R_{\la}((x^\prime,\tau_n);q,\tau)=R_{\la}(x^\prime;q,\tau^\prime),
\]
with on the left-hand side the interpolation Laurent polynomial in $n$
variables and on the right-hand side the interpolation Laurent
polynomial in $n-1$ variables.
\item
For $\la\in\La_n^+$ with $\la_n>0$ we have
\[
R_\la(x;q,\tau)=R_{\la-\mathbf{1}}(x;q,q\tau)
\prod_{i=1}^n\frac{(x_i-\tau_n)(x_i^{-1}-\tau_n)}
{(\overline{\la}_i-\tau_n)(\overline{\la}_i^{\,-1}-\tau_n)}.
\]
\end{list}
\end{proposition}

Set $\La_{n,-1}^+:=\emptyset$ and write
\[
\widehat{\La}_{n,d}^+:=
\La_{n,d}^+\setminus\La_{n,d-1}^+
\]
for the partitions of length at most $n$ and weight $d$.
Then $\{[m_\la]_d \mid \la\in\widehat{\La}_{n,d}^+\}$ is a linear basis of
$\mathcal{G}_{n,d}^{W_n}$.
Furthermore, by Proposition \ref{th:3},
\[
\{[R_\la]_d\mid \la\in\widehat{\La}_{n,d}^+\}
\]
is also a linear basis of $\mathcal{G}_{n,d}^{W_n}$. In particular,
\[
\textup{deg}(R_\la(x;q,\tau))=|\la|\qquad (\la\in\La_n^+).
\]
The following important property is less immediate.
\begin{proposition}
\label{th:8}
Let $\tau\in\FST_n$ and $\la\in\widehat{\La}_{n,d}^+$. 
The coefficient $c_{\la,\la}$ in the 
expansion 
$R_\la(x;q,\tau)=\sum_{\mu\in\La_{n,d}^+}c_{\la,\mu}m_\mu(x)$
($c_{\la,\mu}\in\CC$) is nonzero.
\end{proposition}
%
\Proof
It suffices to show that
the coefficient of $[m_\la]_d$ in the expansion of $[R_\la]_d$ in
terms of the linear basis $\{[m_\mu]_d\mid \mu\in\widehat{\La}_{n,d}^+\}$
of $\mathcal{G}_{n,d}^{W_n}$ is nonzero. We prove this
 by induction on $n+d$.
For $n=1$ the result follows from Example \ref{voorbeeld}.
To prove the induction step we need to consider two cases.

If $\la_n>0$ then $d\geq n$ and 
\begin{equation}\label{Gridentity}
[R_\la]_d=\left(\prod_{i=1}^n\frac{\tau_n}{(\overline{\la}_i-\tau_n)
(\tau_n-\overline{\la}_i^{\,-1})}\right)
[R_{\la-\mathbf{1}}]_{d-n}[m_{\mathbf{1}}]_n
\end{equation}
in $\mathcal{G}(\FSP_n^{W_n})$ by Proposition \ref{th:6}{\bf b)}. The
result now immediately follows from the induction hypothesis.

If $\la_n=0$ then first consider $R_\la(x^\prime;q,\tau^\prime)$.
By the induction
hypothesis, $m_\la(x^\prime)$
occurs with nonzero coefficient in the linear expansion of
$R_\la(x^\prime;q,\tau^\prime)$ in the basis
$\{m_{\mu}(x^\prime)\mid\mu\in\La_{n-1,d}^+\}$ of $\FSP_{n-1,d}^{W_{n-1}}$.
By the Proof of
Lemma \ref{th:4}{\bf c)}, there exists a $g\in\FSP_{n,d}^{W_n}$ such that
\[ 
g(x^\prime,\tau_n)=R_\la(x^\prime;q,\tau^\prime)
\]
and such that $m_\la(x)$ occurs with nonzero coefficient in the linear
expansion of $g(x)$ in the
basis $\{m_{\mu}(x)\mid \mu\in\La_{n,d}^+\}$ of $\FSP_{n,d}^{W_{n}}$.
Fixing this choice of $g$, there exists by (the proof of)
Proposition~\ref{th:3} a unique $h\in\FSP_{n,d-n}^{W_n}$
such that
\begin{equation}\label{exp}
R_\la(x;q,\tau)=g(x)+h(x)\prod_{i=1}^n(x_i-\tau_n)(x_i^{-1}-\tau_n)
\end{equation}
in $\FSP_{n,d}^{W_n}$ (if $d<n$, then \eqref{exp} should be read as
$R_\la(x;q,\tau)=g(x)$ and the proof below goes through with the
obvious adjustments). Hence
\[
[R_\la]_d=[g]_d+(-\tau_n)^n[h]_{d-n}[m_{\mathbf{1}}]_n
\]
in $\mathcal{G}(\FSP_n^{W_n})$, and the result follows from the fact
that the linear expansion of
$[h]_{d-n}[m_{\mathbf{1}}]_n\in\mathcal{G}_{n,d}^{W_n}$ in the basis
$\{[m_\mu]_d\mid \mu\in\widehat{\La}_{n,d}^+\}$ of
$\mathcal{G}_{n,d}^{W_n}$ only involves the basis elements $[m_\mu]_d$
with $\mu_n>0$.
\qed
\bPP
\begin{example}\label{voorbeeld}
If $n=1$ then the interpolation parameter $\tau\in\FST_1$ is a complex
number $s\in\CC^*$ satisfying $s^2\not\in q^{\ZZ}$.  We denote the corresponding
symmetric interpolation Laurent polynomial $R_{(\ell)}(x;q,\tau)$ in one
variable $x$ by $R_\ell(x;q,s)$ ($\ell\in\ZZ_{\ge0}$). 
Then
\begin{equation}
R_\ell(x;q,s)=\frac{(sx,sx^{-1};q)_\ell}{(q^\ell s^2,q^{-\ell};q)_\ell}.
\label{eq:19}
\end{equation}
and the coefficient of $m_\ell(x)$ in the linear expansion
of $R_\ell(x;q,s)$ with respect to the basis
$\{m_k(x) \mid 0\leq k\leq\ell\}$ of $\FSP_{1,\ell}^+$ is 
\[
\frac{(-s)^{\ell}q^{\ell(\ell-1)/2}}
{\bigl(q^{\ell}s^2,q^{-\ell};q\bigr)_{\ell}}.
\]
\end{example}
\vspace{.3cm}
Write $\rho=(\rho_1,\ldots,\rho_n)$ with $\rho_i:=t^{n-i}$, and $st^\rho:=(st^{\rho_1},\ldots,st^{\rho_n})$ for $s,t\in\CC^*$.
 Note that $\rho+(1,\ldots,1)=\frac{1}{2}\sum_{\alpha\in R^+}\alpha$. Furthermore, $st^\rho\in\FST_n$ implies that $t\not\in q^{\ZZ\setminus\{0\}}$ 
and $s^2\not\in q^\ZZ$.

\begin{definition} (cf. \cite{O98}).
Let $s,t\in\CC^*$ such that $st^\rho\in\FST_n$. Then we call
\begin{equation*}
R_\la(x;q,s,t):=R_\la(x;q,st^\rho)
\end{equation*}
the $BC_n$-type interpolation Macdonald polynomial of degree $\la\in\La_n^+$.
\end{definition}
The specialization $\tau=st^\rho$ of the parameters $\tau$ is called the {\it principal specialization}.
By Definition \ref{th:5} and \eqref{eq:8}, 
$R_\la(x;q,s,t)$ is the $\textup{BC}_n$-symmetric Laurent
polynomial of degree $|\la|$ such that
\begin{equation}
R_\la(q^\mu st^\rho;q,s,t)=\de_{\la,\mu}\qquad
(\mu\in\La_n^+,\;|\mu|\le|\la|).
\label{eq:30}
\end{equation}
It is related to Okounkov's polynomial $P_\la^*(x;q,s,t)$ from \cite[Definition 1.1]{O98} by the formula
\begin{equation*}
R_\la(x;q,s,t)=\frac{P_\la^*(xt^{-\rho}s^{-1};q,t,s)}
{P_\la^*(q^\la;q,t,s)}\,.
\end{equation*}
For the $BC_n$-type interpolation Macdonald polynomials, part a) of Proposition \ref{ElementaryProp} is \cite[Prop. 2.2]{O98}, part b) of Proposition
\ref{ElementaryProp} is \cite[Prop. 2.1]{O98}, and Proposition \ref{th:8} is a special case of \cite[Cor. 5.4]{O98}.

The $BC_n$-type interpolation Macdonald polynomials form a distinguished class of $BC_n$-symmetric Laurent interpolation polynomials. They satisfy various special properties, such as the
{\em extra vanishing property} 
\begin{equation}
R_\la(q^\mu st^\rho;q,s,t)=0\quad
\mbox{if}
\quad\la\nsubseteq\mu,
\label{extravansym}
\end{equation}
and they admit explicit binomial, combinatorial and integral formulas, see \cite{O98,O98p2}. The combinatorial formula \cite[Thm. 5.2]{O98} allows to obtain more precise information on the expansion components of $R_\la(x;q,s,t)$ in symmetric monomials, while the binomial formula \cite[Thm. 7.1]{O98} provides the explicit expansion of Koornwinder polynomials in terms of $BC_n$-type interpolation Macdonald polynomials.
%
\begin{remark}\label{interpolationgrid}
The interpolation grid for the $BC_n$-type interpolation Macdonald polynomials naturally appears in the theory of Koornwinder polynomials in the following way.
Koornwinder polynomials are the $BC_n$-symmetric Laurent polynomial eigenfunctions of the commuting Koornwinder--van Diejen $q$-difference operators \cite{K92,vD95}, depending on five parameters $a,b,c,d,t$. 
These operators generate a commutative algebra isomorphic to $\FSP_n^{W_n}$ through the Harish-Chandra isomorphism (cf. \cite[\S 2]{LS}). Through this isomorphism, the eigenvalues of the Koornwinder-van Diejen $q$-difference operators are described by the evaluation morphisms $\FSP_n^{W_n}\rightarrow\CC$, $p\mapsto p(q^\la st^\rho)$ ($\la\in\La_n^+$), where $s=\sqrt{q^{-1}abcd}$.
\end{remark}
%
\section{Interpolation theorem for nonsymmetric Laurent
polynomials}
\label{sec:4}
We extend definition \eqref{eq:8} of the interpolation points
$\overline{\mu}$
from $\mu\in\La_n^+$ to $\mu\in\Lan$ as follows. Put
$\tau\in\FST_n$, with $\FST_n$ defined by \eqref{eq:9}. 
For $\al\in\Lan$ we define $\wb\al=(\wb\al_1,\ldots,\wb\al_n)\in\CC^n$
by
\begin{equation}
\wb\al_i:=q^{\al_i}\big(\tau_{\pi_\al^{-1}(i)}\big)^{\sgn(\al_i)}.
\label{eq:2}
\end{equation}
Here $\pi_\al$ is as in Lemma \ref{th:2}.
We write $\wb\al=\wb\al(q,\tau)$
and $\wb\al_i=\wb\al_i(q,\tau)$ if we 
need to emphasize the dependence of $\wb\al$ on the parameters.

 Recall the actions  \eqref{eq:20}, \eqref{eq:21}, \eqref{eq:22} of $W_n$ 
on $\Lan$, $(\CC^*)^n$ and $\FSP_n$, respectively.
The resulting
action of $W_n$ on the interpolation points
$\wb\al\in (\CC^*)^n$
can be described as follows.
%
\begin{lemma}\label{th:9}
Let $\al\in\Lan$.
\alphlist
\item 
Let $j\in [1,n]$. If $s_j\al\neq\al$ then $\bm{s_j}\wb\al=\wb{s_j\al}$. 
\item
If $j\in [1,n)$ and 
$\al_j=\al_{j+1}$ 
then
\begin{equation}\label{explicitequal}
\wb\al_j/\wb\al_{j+1}=
\bigl(\tau_{\pi_\al^{-1}(j)}/\tau_{\pi_{\al}^{-1}(j)+\sgn(\al_j)}\bigr)^{
\sgn(\al_j)}.
\end{equation}
\item
If 
$\al_n=0$ then $\wb\al_n=\tau_n$.
\end{list}
\end{lemma}
%
\begin{remark}
In case of the special specialisation $\tau=st^{^\rho}$ part a) of Lemma \ref{th:9} was observed by Sahi \cite[Proof of Thm.~5.3]{Sahi99}, and
parts b) and c) were observed in \cite[Remark 4.7]{St}.
\end{remark}
\noindent
{\bf Proof of Lemma \ref{th:9}.}\quad
First we prove a).
For $j=n$ this reduces by \eqref{eq:2} and the assumption $s_n\al\ne\al$ to
showing that
$\pi_{s_n\al}^{-1}=\pi_\al^{-1}$.
It follows immediately from Lemma \ref{th:2}
that these two permutations are equal.
For $j<n$ the statement of a) reduces by \eqref{eq:2} to showing that
$\pi_{s_j\al}^{-1}=\pi_\al^{-1}\circ s_j$ if $s_j\al\ne\al$.
Also the equality of these two permutations under the given condition
follows immediately from Lemma \ref{th:2}.

Similarly the proofs of b) and c) reduce by \eqref{eq:2} to respectively
showing that
$\pi_\al^{-1}(j+1)=\pi_\al^{-1}(j)+\sgn\al_j$ and
$\pi_\al^{-1}(n)=n$. Both statements immediately follow from Lemma \ref{th:2}
taking into account the assumption.

\begin{corollary}\label{cor:9}
Let $\tau\in\FST_n$. \\
{\bf a)} For all $\al\in\Lan$ we have $\wb\al=\bm{w_\al}\wb{\al^+}$.\\
{\bf b)} The map $\Lan\rightarrow (\CC^*)^n$, $\al\mapsto \wb\al$,
is injective.\\
{\bf c)} We have $\wb\al^{\be^j}\not=1$ for $j\in [1,n]$ and $\al\in\Lan$ such that $s_j\al\not=\al$.
\end{corollary}
\Proof
{\bf a)} By Definition \ref{def:2.1}, equation \eqref{eq:8},
Lemma \ref{th:2}a) and equation \eqref{eq:2} we have
\begin{align*}
&(\bm{w_\al} \overline{\al^+})_i=(\bm{\si_\al}\pi_\al\overline{\al^+})_i
=\big(\pi_\al\overline{\al^+}\,\big)_i^{\sgn\al_i}
=\big(\overline{\al^+}\,\big)_{\pi_\al^{-1}(i)}^{\sgn\al_i}\\
&=\big(q^{\al_{\pi_\al^{-1}(i)}}\tau_{\pi_\al^{-1}(i)}\big)^{\sgn\al_i}
=\big(q^{|\al_i|}\big)^{\sgn\al_i}\big(\tau_{\pi_\al^{-1}(i)}\big)^{\sgn\al_i}
=q^{\al_i}\big(\tau_{\pi_\al^{-1}(i)}\big)^{\sgn\al_i}
=\overline\al_i\,.
\end{align*}
{\bf b)} This follows from the explicit expression \eqref{eq:2} using the fact
that $q$ is not a root of unity and that $\tau\in\FST_n$ (see \eqref{eq:9}).\\
{\bf c)}
By part b) of the Corollary and Lemma \ref{th:9}a) we have for $s_j\al\not=\al$
that ${\mathbf{s}_j}\wb\al=\wb{s_j\al}\ne\wb\al$.
Hence $\wb\al^{\be^j}=\wb\al_j/\wb\al_{j+1}\ne 1$
for $j<n$ and $\wb\al^{\be^j}=\wb\al_n^2\ne 1$ for $j=n$.\qed
\bPP
For $d\in\ZZ_{\ge0}$ and $I=\{i_1,\ldots,i_k\}\subseteq[1,n]$
put
\begin{align*}
\La_{n,d}&:=\{\mu\in\Lan\mid |\mu|\le d\},\\
R(n,d,I)&:=\{\al\in \La_{n,d}\mid \al_j\ne0\;
\mbox{for all $j$ and $\al_j\ne-1$ if
$j\in I^{\mathsf{c}}$}\},\\
T(n,d,I)&:=\{\al\in \La_{n,d}\mid \al_j\ne0\;\mbox{if $j\in I^{\mathsf{c}}$}\}.
\end{align*}
Note that
\begin{equation*}
R\big(n,d,[1,n]\big)=T(n,d,\emptyset),\qquad
T\big(n,d,[1,n]\big)=\La_{n,d}.
\end{equation*}
Furthermore,
\begin{equation*}
R(n,d,I)=\emptyset\quad\mbox{if}\quad d-n<0;\qquad
T(n,d,I)=\emptyset\quad\mbox{if}\quad d-n+|I|<0.
\end{equation*}
%
\begin{proposition}
\label{th:1}
Let $n\in\ZZ_{>0}$, $d\in\ZZ_{\ge0}$, $\tau\in\FST_n$ and suppose that $I=\{i_1,\ldots,i_k\}\subseteq [1,n]$ is a set of cardinality $k$.
\alphlist
\item
For every map $\wb f\colon R(n,d,I)\to\CC$
there exists a Laurent polynomial $f\in\FSP_n$ such that

\begin{equation}\label{Inda}
\begin{split}
f(\wb \al(q,\tau))&=\wb f(\al)\qquad\qquad \forall\,\al\in R(n,d,I),\\
\deg\bigl(x_Jf(x)\bigr)&\leq d-n+k\qquad\forall\, J\subseteq I.
\end{split}
\end{equation}
\item
For every map $\wb f\colon T(n,d,I)\to\CC$
there exists a Laurent polynomial $f\in\FSP_n$ such that

\begin{equation}\label{Indb}
\begin{split}
f(\wb \al (q,\tau))&=\wb f(\al)\qquad\qquad \forall\, \al\in T(n,d,I),\\
\deg\bigl(x_Jf(x)\bigr)&\leq d\qquad\qquad\quad\,\,\, \forall J\subseteq I^c.
\end{split}
\end{equation}
\end{list}
\end{proposition}
\begin{remark}\label{remspecialcases}
Note that statement~b) for $I=\emptyset$ is statement~a) for $I=[1,n]$.
\end{remark}
\noindent {\bf Proof of Proposition \ref{th:1}.}\quad
If $\wb f$ is a map on an empty set then choose $f$ identically zero.
Thus statement a) holds trivially when $d<n$ and statement b) holds
trivially when $d<n-k$.

If $(n,d)=(1,0)$ then statement a), and statement b) for
$I=\emptyset$, are true by the remark in the previous paragraph. For
statement b) with $I=\{1\}$ note that $T(1,0,\{1\})=\{0\}$,
hence we can take $f(x)$ to be the constant polynomial $\overline{f}(0)$. 

Now let $n+d\geq 2$.
Suppose that all the statements of the Proposition, with
$(n,d)$ replaced by $(\widetilde{n},\widetilde{d})$, are true for all subsets $I\subset [1,\widetilde{n}]$ and all $\tau\in\FST_{\widetilde{n}}$
when $\widetilde{n}+\widetilde{d}<n+d$. We will
then successively prove statements a) and b) by induction on the
cardinality $|I|$ of the subset $I\subseteq [1,n]$.
\paragraph{Proof of statement a).}
We may assume that $d\ge n$. 

First consider the case $I=\emptyset$. Note that
\[
R(n,d,\emptyset)=\{\al\in \La_{n,d}\mid\al_j\ne0,-1\;\mbox{for all $j$}\}.
\]
Fix a map $\wb f: R(n,d,\emptyset)\rightarrow\CC$ and $\tau\in\FST_n$.
We prove the existence of a Laurent polynomial $f\in\FSP_{n,d-n}$ such that 
$f(\overline{\al}(q,\tau))=f(\al)$ for all $\al\in R(n,d,\emptyset)$
by solving a related interpolation problem
on the set $T(n,d-n,[1,n])=\La_{n,d-n}$, using statement b) with shifted parameters $q\tau\in\FST_n$.

Consider for this the bijection 
\begin{equation}\label{bij}
 \al\mapsto\be\colon R(n,d,\emptyset)\overset{\sim}{\longrightarrow} T(n,d-n,[1,n]),
 \qquad
\be_i:=\al_i-\sgn(\al_i).
\end{equation}
Note that $\sgn(\be_i)=\sgn(\al_i)$ and, using \eqref{altdef},
$\pi_\be=\pi_\al$.
Hence, by
\eqref{eq:2},
$\wb\al(q,\tau)=\wb\be(q,q\tau)$. By the induction hypothesis, statement~b) with $\tau$ replaced by $q\tau$ is valid for the function $\be\mapsto\overline{f}(\al)$ ($\be\in T(n,d-n,[1,n])$).  Hence there exists a
$f\in\FSP_{n,d-n}$ such that $f(\wb\be(q,q\tau))=\overline{f}(\al)$
for all $\be\in T(n,d-n,[1,n])$.  But $f\in\FSP_{n,d-n}$ then also
satisfies $f(\overline{\al}(q,\tau))=f(\al)$ for all
$\al\in R(n,d,\emptyset)$, which completes the proof of statement a) for $I=\emptyset$.

Now let $k>0$ and assume that statement a) is true for 
all functions $\wb f^{\,\vee}\colon
R(n^\vee,d^\vee,I^\vee)\rightarrow \CC$ when
$n^\vee+d^\vee\leq n+d$, $\tau^\vee\in\FST_{n^\vee}$ and with
$I^\vee\subseteq [1,n^\vee]$ of cardinality $<k$.
(Note that for $n^\vee+d^\vee<n+d$ this assumption already holds by our
earlier induction hypothesis.)
Let $I=\{i_1,\ldots,i_k\}\subseteq [1,n]$ be a set of cardinality $k$,
$\tau\in\FST_n$,
and consider a function $\wb f\colon R(n,d,I)\rightarrow\CC$. We prove
the existence of an interpolation Laurent polynomial $f\in\FSP_n$ 
satisfying \eqref{Inda}
by splitting the interpolation problem in two pieces. For this we use
the disjoint union
\[
R(n,d,I)=R^\vee(n,d,I)\sqcup R(n,d,I\backslash\{i_1\})
\]
with
\[
R^\vee(n,d,I):=\{\al\in R(n,d,I)\mid\al_{i_1}=-1\}.
\]

The first step is to prove the existence of a Laurent polynomial $g\in\FSP_n$ such that
\begin{equation}\label{Indc}
\begin{split}
g(\wb\al(q,\tau))&=\wb f(\al)\qquad\qquad\qquad\forall\, \al\in R^\vee(n,d,I),\\
\deg\bigl(g(x)x_K)&\leq d-n+k-1\qquad\, \forall\, K\subseteq I\setminus\{i_1\}.
\end{split}
\end{equation}
For $n=1$ we have $d\geq 1$ and $I=\{1\}$, hence
$R^\vee(1,d,\{1\})=\{-1\}$. In this case we can take $g(x)$ to be
the constant polynomial $\overline{f}(-1)$. Assume that $n>1$. In this
case we solve the interpolation problem \eqref{Indc}
by rewriting it as  
an interpolation problem for a function on
$R(n-1,d-1,J)$ with $J:=\{i_2-1,i_3-1,\ldots,i_k-1\}\subseteq [1,n-1]$.

Consider the bijection
\[
\al\mapsto\ga\colon R^\vee(n,d,I)\overset{\sim}{\longrightarrow} R(n-1,d-1,J)
\]
with
$\ga:=\al^{(i_1)}=(\al_1,\ldots,\al_{i_1-1},\al_{i_1+1},\ldots,\al_n)$. In
other words, $\ga_i=\al_{i^\vee}$ for $i\in[1,n)$ with $i^\vee:=i$
if $i<i_1$ and $i^\vee:=i+1$ if $i_1\le i<n$. 

The interpolation points behave under this bijection in the following manner.
By the explicit description of $\pi_\al$ (see Lemma \ref{th:2}) we have
$\pi_\al^{-1}(i_1)=n$ and $\pi_\al^{-1}(i^\vee)=\pi_\ga^{-1}(i)$
($i\in [1,n)$) for $\al\in R^\vee(n,d,I)$.
Then, by \eqref{eq:2}, we have
\begin{equation}\label{hulp}
\begin{split}
\wb\al_{i^\vee}(q,\tau)&=\wb\ga_i(q,\tau')\quad (1\leq i<n),\\
\wb\al_{i_1}(q,\tau)&=q^{-1}\tau_n^{-1}
\end{split}
\end{equation}
for $\al\in R^\vee(n,d,I)$.

Consider the function $\overline{g}^{\vee}: R(n-1,d-1,J)\rightarrow\mathbb{C}$, defined by $\overline{g}^{\vee}(\ga):=\overline{f}(\al)$.
Since $|J|<k$ the induction hypothesis
(either  the one on the sum of the number of variables and
the weight, or the one on the size of the subset)
implies that statement a) is true for $\overline{g}^\vee$ and $\tau^\prime\in\FST_{n-1}$. Hence there exists a Laurent polynomial
$g^\vee\in\FSP_{n-1}$ such that
$g^\vee(\wb\ga(q,\tau'))=\wb f(\al)$ for all $\ga\in R(n-1,d-1,J)$,
satisfying the degree conditions
$\deg\bigl(g^\vee(x')x_K\bigr)\leq d-n+k-1$ for all $K\subseteq J$.
Define $g\in\FSP_n$ by $g(x):=g^\vee(x^{(i_1)})$, then it follows
that $g$ satisfies \eqref{Indc}. This completes the first step.

As a second step we add an appropriate term to $g$ to obtain the desired interpolation properties for the full set $R(n,d,I)$.
Note that Laurent polynomials $f$ of the form
\begin{equation}
f(x)=g(x)+(x_{i_1}^{-1}-q\tau_n)h(x)
\label{eq:3}
\end{equation}
with $h\in\FSP_n$ all satisfy
$f(\wb\al(q,\tau))=\wb f(\al)$ 
for $\al\in R^\vee(n,d,I)$ in view of \eqref{hulp}. The interpolation property $f(\wb\al(q,\tau))=\wb f(\al)$
is then also satisfied for
$\al \in R(n,d,I\backslash\{i_1\})$ if
\begin{equation}\label{hinter}
h(\wb\al(q,\tau))=
\frac{\wb f(\al)-g(\wb\al(q,\tau))}{\wb\al_{i_1}(q,\tau)^{-1}-q\tau_n}\qquad
\big(\al \in R(n,d,I\backslash\{i_1\})\big).
\end{equation}
Note that the right-hand side is well defined, since the conditions
\eqref{eq:9}
on the parameters together with \eqref{eq:2} and the fact that $\al_{i_1}\ne-1$
imply that the denominator is nonzero.

Due to the induction hypothesis, we are allowed to apply statement a)
to the function
\[
\overline{h}(\al):=
\frac{\wb f(\al)-g(\wb\al(q,\tau))}{\wb\al_{i_1}(q,\tau)^{-1}-q\tau_n}\qquad
\big(\al \in R(n,d,I\backslash\{i_1\})\big).
\]
This gives a Laurent polynomial $h\in\FSP_n$ fulfilling \eqref{hinter} and
satisfying the degree conditions $\deg\bigl(h(x)x_K\bigr)\leq d-n+k-1$
for all $K\subseteq I\setminus\{i_1\}$.  Then $f$ given by
\eqref{eq:3} satisfies $f(\wb\al(q,\tau))=\wb f(\al)$ for all $\al\in
R(n,d,I)$. Furthermore, by the degree properties of $g$ and $h$ the
degree conditions $\deg(f(x)x_J)\leq d-n+k$ for all $J\subseteq I$ are
satisfied. Hence $f$ satisfies \eqref{Inda}, as desired.
\paragraph{Proof of statement b).}
The proof is along the same lines as the proof of statement a), but there are 
subtle differences in the combinatorics.

We may assume that $d\ge n-k$.
If $I=\emptyset$ then the statement is correct due to Remark \ref{remspecialcases}.

Let $k>0$ and assume that statement b) is true
for all functions
$\wb f^{\,\wedge}\colon R(n^\wedge,d^\wedge,I^\wedge)\rightarrow\CC$ when
$n^\wedge+d^\wedge\leq n+d$, $\tau^\wedge\in\FST_{n^\wedge}$ and
$I^\wedge\subseteq [1,n^\wedge]$ is a set of cardinality $<k$.
(Note that for $n^\vee+d^\vee<n+d$ this assumption already holds by our
earlier induction hypothesis.)
Let $I=\{i_1,\ldots,i_k\}\subseteq [1,n]$ be a set of cardinality $k$
and consider a function $\wb f\colon T(n,d,I)\rightarrow\CC$. 
We have to prove the existence of
an interpolation Laurent polynomial $f\in\FSP_n$ satisfying \eqref{Indb}.

Consider this time
the decomposition 
\[
T(n,d,I)=T^\wedge(n,d,I)\sqcup T(n,d,I\backslash\{i_k\})
\]
with
\[
T^\wedge(n,d,I):=\{\al\in T(n,d,I)\mid\al_{i_k}=0\}.
\]
We claim that there exists a 
$g\in\FSP_n$ such that
\begin{equation}\label{Indd}
\begin{split}
g(\wb\al(q,\tau))&=\wb f(\al)\qquad\,\,\, \forall\,\al\in T^\wedge(n,d,I),\\
\deg\bigl(g(x)x_K\bigr)&\leq d\qquad\qquad \forall\, K\subseteq I^c.
\end{split}
\end{equation}
For $n=1$ we have $d\geq 1$ and $I=\{1\}$, hence
$T^\wedge(1,d,\{1\})=\{0\}$ and we can take $g(x)$ to be the constant
polynomial equal to $\wb f(0)$. For $n>1$ consider the
bijection
\[
\al\mapsto\de\colon T^\wedge(n,d,I)\overset{\sim}{\longrightarrow} T(n-1,d,I\setminus\{i_k\})
\]
with $\de:=\al^{(i_k)}$.
In other words, $\de_i:=\al_{i^\wedge}$ ($i\in [1,n)$) with $i^\wedge$
defined by $i^\wedge:=i$ if $i<i_k$ and $i^\wedge:=i+1$ if
$i_k\leq i<n$. As in the proof of assumption a) one then shows that
\begin{equation}\label{hulp2}
\begin{split}
\wb\al_{i^\wedge}(q,\tau)&=\wb\de_{i}(q,\tau^\prime)\qquad i\in [1,n),\\
\wb\al_{i_k}(q,\tau)&=\tau_n
\end{split}
\end{equation}
for $\al\in T^\wedge(n,d,I)$. 
By the induction hypothesis (either
the induction hypothesis on the sum of the number of variables and
the weight, or the induction hypothesis on the size of the
subset), there exists a $g^\wedge\in\FSP_n$ such that
$g^\wedge(\wb\de(q,\tau'))=\wb f(\al)$ for
$\de\in T(n-1,d,I\setminus\{i_k\})$ and satisfying the degree conditions
$\deg\bigl(g^\wedge(x')x_K\bigr)\le d$ for all
$K\subseteq(I^c\cup\{i_k\})\cap [1,n-1]$.
Then $g\in\FSP_n$, defined by $g(x):=g^\wedge(x^{(i_k)})$, 
satisfies \eqref{Indd}.
 
Now define $\wb h\colon T(n,d,I\setminus\{i_k\})\rightarrow\CC$ by
\[
\wb h(\al):=
\frac{\wb f(\al)-g(\wb\al(q,\tau))}{\wb\al_{i_k}(q,\tau)-\tau_n}\qquad
(\al\in T(n,d,I\setminus\{i_k\})).
\]
Note that the right-hand side is well defined, since the
conditions \eqref{eq:9} on the parameters together with \eqref{eq:2}
and the fact that $\al_{i_k}\ne0$ imply that the denominator is nonzero.
By the induction hypothesis, there exists a $h\in\FSP_n$ such that
$h(\wb\al(q,\tau))= \wb h(\al)$ for all
$\al\in T(n,d,I\setminus\{i_k\})$ which satisfies the degree conditions
$\deg\bigl(h(x)x_K\bigr)\leq d$ for all $K\subseteq
I^c\cup\{i_k\}$. Furthermore, with this choice of $h$ and \eqref{hulp2} it is clear that
\[
f(x):=g(x)+(x_{i_k}-\tau_n)h(x)
\]
satisfies the desired interpolation property
$f(\wb\al(q,\tau))=f(\al)$ for all $\al\in T(n,d,I)$.
By the degree conditions on $g(x)$ and $h(x)$, we have
$\deg\bigl(f(x)x_J\bigr)\leq d$ for all $J\subseteq I^c$, which
completes the proof of statement b).
\qed
%
\begin{theorem}
\label{th:7}
Let $\tau\in\FST_n$.
For every map $\wb f\colon \La_{n,d}\to\CC$
there exists a unique Laurent polynomial $f\in\FSP_{n,d}$
such that $f(\wb \al(q,\tau))=\wb f(\al)$ for all $\al\in \La_{n,d}$.
\end{theorem}
%
\Proof
Denote by $\mathcal{F}_{n,d}^{q,\tau}$ the space of complex-valued functions on $\mathcal{S}_{n,d}^{q,\tau}:=\{\wb\al(q,\tau)\mid\al\in \La_{n,d}\}$.
Then, by Corollary \ref{cor:9}b), $\mathcal{F}_{n,d}^{q,\tau}$ has dimension
$|\La_{n,d}|$.
Define the linear map
$\phi_{n,d}^{q,\tau}: \FSP_{n,d}\rightarrow\mathcal{F}_{n,d}^{q,\tau}$ by
\[
\phi_{n,d}^{q,\tau}(f):=f\vert_{\mathcal{S}_{n,d}^{q,\tau}}.
\]
Proposition \ref{th:1}b) with $I=[1,n]$ implies that $\phi_{n,d}^{q,\tau}$ is
surjective. Then $\phi_{n,d}^{q,\tau}$ is also injective, since both vector
spaces $\FSP_{n,d}$ and $\mathcal{F}_{n,d}^{q,\tau}$ are of dimension
$|\La_{n,d}|$.
Hence $\phi_{n,d}^{q,\tau}$ is a linear isomorphism, which implies the
theorem.\qed
\bPP
In the remainder of this section we fix $\tau\in\FST_n$ and write $\overline{\al}:=\overline{\al}(q,\tau)$ for $\al\in\La_n$.
In view of Theorem \ref{th:7}, we can give the following
definition.
%
\begin{definition}
\label{def:5}
The (nonsymmetric) interpolation Laurent polynomial of
degree $\al\in\La_n$ is the unique Laurent polynomial
 $G_\al(x;q,\tau)$ 
 in $n$ variables
$x$ of degree at most $|\al|$
such that $G_\al(\wb\al;q,\tau)=1$ and

\begin{equation*}
G_\al(\wb\be;q,\tau)=0\qquad \forall\,\be\in\La_{n,|\al|}\setminus\{\alpha\}.
\end{equation*}
\end{definition}
Theorem \ref{th:7} implies that
$\{G_\al(x;q,\tau)\mid\al\in\La_{n,d}\}$ is a linear basis of $\FSP_{n,d}$.
\begin{example}\label{voorbeeld2}
Recall from Example \ref{voorbeeld} that for $n=1$ the interpolation parameter $\tau\in\FST_1$ is
given by a complex number $s$ satisfying $s^2\not\in q^\ZZ$.
We write $G_{(\ell)}(x;q,\tau)$
for $\ell\in\ZZ$ by $G_\ell(x;q,s)$. 
Then the Laurent polynomial $G_\ell(x;q,s)$ has degree at most $|\ell|$
and is
characterized by the equations $G_\ell(q^k s^{\sgn(k)};q,s)=\de_{\ell,k}$
for $k\in\ZZ$ with $|k|\leq |\ell|$. It follows that
\begin{equation}
\begin{split}
G_\ell(x;q,s)&=\frac{\big(qsx,sx^{-1};q\big)_\ell}{\big(q^{1+\ell}s^2,q^{-\ell};q\big)_\ell},
\qquad\qquad\qquad\quad \ell\in\ZZ_{\geq 0},\\
G_{-\ell}(x;q,s)&=\frac{q^\ell sx\big(qsx;q\big)_{\ell-1}\big(sx^{-1};q\big)_{\ell+1}}
{\big(q^\ell s^2;q\big)_{\ell+1}\big(q^{1-\ell};q\big)_{\ell-1}},\qquad  \ell\in\ZZ_{>0}.
\end{split}
\label{eq:17}
\end{equation}
Furthermore, for $n>1$,
\[
G_\al(x;q,\mathbf{s}):=\prod_{i=1}^nG_{\al_i}(x_i;q,s)\qquad (\al\in\La_n)
\]
with (recall) $\mathbf{s}=(s,\ldots,s)\in\FST_n$.
\end{example}

As in Section \ref{sec:3}, one concludes from Theorem \ref{th:7}
that $\{[G_\al]_d\mid \al\in\widehat{\La}_{n,d}\}$ is a linear basis of
$\mathcal{G}_{n,d}$, where
$\widehat{\La}_{n,d}:=\La_{n,d}\setminus\La_{n,d-1}$ and $\La_{n,-1}:=\emptyset$.
In particular,
\[
\deg\bigl(G_\al(x;q,\tau)\bigr)=|\al|\qquad (\al\in\La_n).
\]
Recall from Proposition \ref{th:8} that the coefficient of $m_\la(x)$ in the linear expansion of the symmetric interpolation Laurent polynomial 
$R_\la(x;q,\tau)$ in symmetric monomials $m_\mu(x)$ ($\mu\in\La_n^+$) is nonzero. For the nonsymmetric interpolation Laurent polynomial $G_\al(x;q,\tau)$ we have the
following result.
%
\begin{lemma}\label{lem:rf}
Let 
$\al\in\Lan$. The coefficients $c^\al_\ga(q,\tau)\in\CC$ in the linear expansion
\begin{equation}
G_\al(x;q,\tau)=\sum_{\ga\in \La_{n,|\al|}}c^\al_{\ga}(q,\tau)\, x^\ga 
\label{eq:18}
\end{equation}
are rational functions in the variables $q,\tau_1,\ldots,\tau_n$. The rational function in $q,\tau_1,\ldots,\tau_n$
representing
$c^\al_\al(q,\tau)$ for $q$ not a root of unity and $\tau\in\FST_n$, is nonzero.
\end{lemma}
%
\Proof
We will prove that the $c^\al_\ga(q,\tau)$ are rational in $q,\tau$ by
considering them for fixed $\al$ as solutions of a linear system with rational
coefficients.
We use the notations introduced in the proof of Theorem \ref{th:7}. Let $d\in\ZZ_{\geq 0}$ and write $m:=|\La_{n,d}|$. Identify
$\Lambda_{n,d}$ with $[1,m]$ by fixing an enumeration of the elements in  $\Lambda_{n,d}$. This provides 
vector space identifications
\[
\sum_{\ga\in\La_{n,d}}d_\ga x^\ga\mapsto \bigl(d_\ga\bigr)_{\ga}:\, \FSP_{n,d}\overset{\sim}{\longrightarrow} \CC^m,
\qquad f\mapsto \bigl(f(\wb\ga(q,\tau))\bigr)_{\ga}:\, \mathcal{F}_{n,d}^{q,\tau}\overset{\sim}{\longrightarrow}\CC^m.
\]
The linear isomorphism $\phi_{n,d}^{q,\tau}:\FSP_{n,d}\overset{\sim}{\longrightarrow}\mathcal{F}_{n,d}^{q,\tau}$ is then represented by the invertible matrix
\[
A(q,\tau):=\bigl(\wb\be(q,\tau)^\ga\bigr)_{\be,\ga}\in\textup{GL}_m(\CC).
\]
Let $\{e^\al\mid\,\al\in\La_{n,d}\}$ be the standard basis of $\CC^m$. Then Definition \ref{def:5} implies that
\[
\bigl(c^\al_\ga(q,\tau)\bigr)_\ga=A(q,\tau)^{-1}e_\al\qquad \forall\,\al\in\La_{n,d},
\]
and Example \ref{voorbeeld2} shows that $c^\al_\al(q,\mathbf{s})\not=0$ if $s\in\CC^*$ and $s^2\not\in q^\ZZ$.
The result now follows from the fact that the matrix coefficients of $A(q,\tau)$ are rational functions in
$q,\tau_1,\ldots,\tau_n$.
\qed
\bPP
For $n=1$ we find
 from \eqref{eq:19} and \eqref{eq:17} that
\[
R_\ell(x;q,s)=G_\ell(x;q,s)+G_{-\ell}(x;q,s)
\qquad(\ell\in\ZZ_{\geq 0}).
\]
This generalizes to arbitrary $n\geq 1$ as follows.
\begin{theorem}\label{symmversusnonsymm}
Let $R_\la(x)=R_\la(x;q,\tau)$ and
$G_\al(x)=G_\al(x;q,\tau)$ be the symmmetric and nonsymmetric
interpolation polynomials as given by
Definitions \ref{th:5} and
\ref{def:5}, respectively. 

For $\la\in\La_n^+$ and $\al\in\Lan$ we have
\begin{equation}\label{eq:101}
\begin{split}
R_\la(\wb\al)&=R_\la(\wb{\al^+}),\\
R_\la(x)&=\sum_{\be\in W_n\la}G_\be(x).
\end{split}
\end{equation}
\end{theorem}
 \Proof
The first formula in \eqref{eq:101} follows immediately from part a) of Corollary \ref{cor:9}.

Let $\la\in\La_n^+$ and write $H_\la:=\sum_{\be\in W_n\la}G_\be$.
Then, by the definitions of $R_\la$ and $G_\be$
and by part a) of the theorem, $R_\la$ and $H_\la$ are Laurent
polynomials of degree at most $|\la|$
satisfying 
\[
R_\la(\wb\al)=\de_{\al^+,\la}=H_\la(\wb\al) \qquad (\al\in\La_{n,|\la|}).
\]
By Theorem \ref{th:7} we conclude that $R_\la=H_\la$.
\qed
\bPP
\begin{remark}
An analogous statement as Theorem \ref{symmversusnonsymm}
holds true for Sahi's \cite{S96} symmetric and
nonsymmetric interpolation polynomials, with essentially the same proof.
\end{remark}

Consider the principal specialization $\tau_i=st^{n-i}$ of
$\tau\in\FST_n$.
The explicit formula \eqref{eq:2} for the interpolation
point $\wb\al$ ($\al\in\Lan$) then takes the form
\begin{equation}
\wb\al_i=q^{\al_i}\big(st^{n-\pi_\al^{-1}(i)}\big)^{\sgn(\al_i)}\qquad
(i\in [1,n]).
\label{eq:23}
\end{equation}
If $s=\sqrt{q^{-1}abcd}$ then 
$\wb\al_i$ \eqref{eq:23} corresponds to the eigenvalue of Noumi's $Y$-operator $Y_i$ for the
nonsymmetric Koornwinder polynomial $E_\al(x;a,b,c,d;q,t)$ of degree
$\al$, see \cite[\S 6]{Sahi99} (compare with Remark \ref{interpolationgrid} for the symmetric theory). 
Note that for $\tau=st^\rho$, the second formula in \eqref{eq:101} gives the expansion of the $BC_n$-type interpolation Macdonald polynomial $R_\la(x;q,s,t)$
in terms of the interpolation Laurent polynomials $G_\be(x;q,s,t)$ ($\be\in W_n\la$). This is an analogue for interpolation polynomials of the 
expansion formula expressing Koornwinder polynomials as linear combination of nonsymmetric Koornwinder polynomials, see \cite[Thm.~6.6]{St}.
\begin{definition}
Let $s,t\in\CC^*$ such that $st^\rho\in\FST_n$. Then we call
\begin{equation*}
G_\al(x;q,s,t):=G_\al(x;q,st^\rho)
\end{equation*}
the nonsymmetric $BC_n$-type interpolation Macdonald polynomial of degree $\al\in\La_n$.
\end{definition}
Observe that $G_\al(x;q,s,t$) is the unique Laurent polynomial of degree
$\leq |\al|$ satisfying $G_\al(\wb\al;q,s,t)=1$ and $G_\al(\wb{\be};q,s,t)=0$ 
($\be\in\La_{n,|\al|}\setminus\{\al\}$), with $\wb\al$ given by \eqref{eq:23}. 

In the following two sections we present first steps towards answering the question whether the nonsymmetric $BC_n$-type interpolation Macdonald polynomials satisfy extra vanishing properties and admit explicit binomial formulas.
For Knop's \cite{Kn97} type $A_{n-1}$ nonsymmetric interpolation Macdonald polynomials, extra vanishing and explicit binomial formulas were derived in \cite{Kn97,S96,SaSt}. Their proofs lean on a generalization of Cherednik's action of the double affine Hecke algebra on polynomials in $n$ variables for which the type $A_{n-1}$ nonsymmetric interpolation Macdonald polynomials are common eigenfunctions of the resulting $Y$-operators.

It is not known whether the nonsymmetric $BC_n$-type interpolation Macdonald polynomials $G_\al(x;q,s,t)$ satisfy extra
vanishing properties. In the Appendix we will present the outcome of computer algebra computations describing extra vanishing
for $G_\al(x;q,s,t)$ when $n=2$ and $|\al|=4$.
\section{The action of Demazure--Lusztig operators}
\label{sec:5}
We introduce an action of the 
type $C_n$ Hecke algebra on the space of Laurent polynomials in $n$ variables, defined in terms of Demazure--Lusztig operators.
We explicitly compute its action on nonsymmetric $BC_n$-type interpolation Macdonald polynomials.
Similar to the type $A_{n-1}$ case in
\cite{Kn97,S96}, the Hecke algebra techniques in this section can only be applied when taking the principal specialization $\tau=st^\rho$.

Recall our notations associated with
root system $C_n$  in Section \ref{sec:2}.
\begin{Definition}[Hecke algebra of type $B_n$ or $C_n$]\label{heckealgebra}
Let $\FSH_n(t,t_n)$ be the complex unital associative algebra with 
 generators $T_1,\ldots,T_n$, parameters $t,t_n\in\CC^*$,
 and defining relations
\begin{equation*}
\begin{split}
&T_iT_{i+1}T_i=T_{i+1}T_iT_{i+1},\qquad\qquad\quad i\in [1,n-2],\\
&T_{n-1}T_nT_{n-1}T_n=T_nT_{n-1}T_nT_{n-1},\\
&T_iT_j=T_jT_i,\qquad\qquad\qquad\quad\quad\,\,\,\,\,\,\, |i-j|>1,\\
&(T_i-t_i)(T_i+1)=0,\qquad\qquad\quad\quad i\in [1,n]
\end{split}
\end{equation*}
with $t_i:=t$ for $i\in [1,n)$.
\end{Definition}
\begin{remark}
The relations $(T_i-t_i)(T_i+1)=0$ are related to the usual
Hecke relations $(\wt T_i-\wt t_i)(\wt T_i+\wt t_i^{-1})=0$
\cite[(4.1.1)]{Ma00} by the substitutions
$\wt t_i=t_i^{1/2}$, $\wt T_i=t_i^{-1/2}T_i$
(Sahi \cite[\S2.3]{Sahi99} has Hecke relations as in \cite{Ma00} with
$\wt t_i=t_i^{1/2}$).
\end{remark}

The trivial one-dimensional representation $\chi$ of $\FSH_n(t,t_n)$
is characterized by $\chi(T_i)=t_i$ for $i\in [1,n]$.
For a reduced expression $w=s_{i_1}\cdots s_{i_r}$ of $w\in W_n$
define $T_w\in \FSH_n(t,t_n)$ by $T_w:=T_{i_1}\cdots T_{i_r}$.
This is independent of the choice of the reduced expression,
see \cite[Proposition 1.15]{IM}.
Define the Hecke symmetrizer of $\FSH_n(t,t_n)$ by
\begin{equation}
C_+:=\sum_{w\in W_n}T_w,
\label{eq:28}
\end{equation}
then  
\begin{equation}
hC_+=\chi(h)C_+=C_+h\qquad(h\in \FSH_n(t,t_n)),
\label{eq:29}
\end{equation}
cf., e.g., \cite[(5.5.7), (5.5.9)]{Ma00}.

Noumi ~\cite{No} introduced a one-parameter family of representations
of $\FSH_n(t,t_n)$ on $\FSP_n$ in terms of Demazure--Lusztig type operators \cite[Proposition~3.6]{L}.
Concretely, it is given by
\begin{equation*}
\begin{split}
T_j&\mapsto t+\frac{x_j-tx_{j+1}}{x_j-x_{j+1}}\,(s_j-1),
\qquad\qquad\qquad j\in [1,n),\\
T_n&\mapsto t_n+\frac{(1-ax_n^{-1})(1-bx_n^{-1})}{1-x_n^{-2}}\,(s_n-1)
\end{split}
\end{equation*}
with $a,b\in\CC$ such that $ab=-t_n$. 
As we shall see in Proposition \ref{expansiontheorem}, the specialization
of the Hecke parameters $t,t_n$ and the representation parameters $a,b$
that is needed for the application to nonsymmetric $BC_n$-type interpolation Macdonald polynomials,
is $t_n=-1$ and $a=s$, $b=s^{-1}$
with $s\in\CC^*$. Noumi's 
representation then takes the following form.
%
\begin{lemma}
\label{th:10}
Let $s,t\in\CC^*$.
The assignments $T_j\mapsto H_j^{(t)}$ ($j\in [1,n)$)
and $T_n\mapsto H_n^{(s)}$ with
\begin{align}
H_j^{(t)}&:=t+\frac{x_j-tx_{j+1}}{x_j-x_{j+1}}\,(s_j-1),
\qquad\qquad j\in [1,n),
\label{eq:24}\\
H_n^{(s)}&:=-1+\frac{(1-sx_n^{-1})(1-s^{-1}x_n^{-1})}{1-x_n^{-2}}\,
(s_n-1)\label{eq:25}
\end{align}
define a one-parameter family of representations
$\pi_s\colon \FSH_n(t,-1)\rightarrow
\textup{End}(\FSP_n)$ on $\FSP_n$.
\end{lemma}
%
In the following lemma we show that $\pi_s$ preserves the degree-filtration on 
$\FSP_{n}$.
\begin{lemma}
\label{th:11}
Let $s,t\in\CC^*$. Then $\FSP_{n,d}$ ($d\in\ZZ_{\ge0}$)
is a $\FSH_n(t,-1)$-submodule of $\FSP_n$
with respect to the action $\pi_s$.
\end{lemma}
%
\Proof
We have to prove that $H_n^{(s)}(x^\al)$ and $H_j^{(t)}(x^\al)$ ($j\in [1,n)$)
are Laurent polynomials of degree at most $|\al|$.
Clearly, it is sufficient to prove the first claim for $n=1$ and the
second claim for $n=2$ and $j=1$. 

A straightforward computation gives that
\begin{equation*}
H_1^{(s)}(x^k)=-\sum_{i=0}^kx^{2i-k}-\sum_{i=1}^kx^{2i-k}+(s+s^{-1})\sum_{i=1}^kx^{2i-k-1}
\end{equation*}
for $k\in\ZZ$, where we use the convention that $\sum_{i=k}^\ell f_i$ is equal to $0$ if $k=\ell+1$ and equal to $-\sum_{i=\ell+1}^{k-1}f_i$ if $k>\ell+1$ (then formally,  $\sum_{i=k}^\ell f_i=\sum_{i=k}^\infty f_i-
\sum_{i=\ell+1}^\infty f_i$ for all $k,\ell\in\ZZ$). Similarly,
\begin{equation*}
H_1^{(t)}(x_1^kx_2^\ell)=-\sum_{i=1}^{k-\ell}x_1^{i+\ell}x_2^{-i+k}+t\sum_{i=0}^{k-\ell}x_1^{i+\ell}x_2^{-i+k}.
\end{equation*}
It follows from these formulas that $H_1^{(s)}(x^k)$ and $H_1^{(t)}(x_1^kx_2^\ell)$ are Laurent polynomials of degree $|k|$
and $|k|+|\ell|$, respectively (for the second case observe that the set
$\{(\xi,\eta)\in\RR^2\mid |\xi|+|\eta|\le d\}$ is convex).\qed
\bPP
The action of Demazure--Lusztig type operators on normalized
nonsymmetric Koornwinder polynomials was determined explicitly in
\cite[Prop.~7.8(ii)]{St}, and on nonsymmetric interpolation Macdonald polynomials in \cite[Lem. 10(1)]{SaSt}. For nonsymmetric $BC_n$-type interpolation Macdonald polynomials
we have the following result.
\begin{proposition}\label{expansiontheorem}
Let $\al\in\Lan$ such that $st^\rho\in\FST_n$.
Write
$G_\al=G_\al(\cdot;q,s,t)$ for the nonsymmetric $BC_n$-type interpolation Macdonald polynomial
of degree $\al$, and $\wb\al$ for
the interpolation point \eqref{eq:23}.
\alphlist
\item
Let $j\in [1,n)$. Then
\begin{equation}\label{f1}
\begin{split}
H_j^{(t)}G_\al&=tG_\al\qquad\qquad\qquad\qquad\qquad\qquad
\qquad \hbox{if } s_j\al=\al,\\
H_j^{(t)}G_\al&=-G_\al+
\frac{\wb{\al}_{j+1}-t\wb{\al}_j}
{\wb{\al}_{j+1}-\wb{\al}_{j}}\,
\big(G_{s_j\al}+G_\al\big)
\qquad \hbox{if } s_j\al\not=\al.
\end{split}
\end{equation}
\item
We have
\begin{equation}\label{f2}
H_n^{(s)}G_\al=-G_\al+
\frac{(1-s\wb{\al}_n)(1-s^{-1}\wb{\al}_n)}
{1-\wb{\al}_n^2}\,
\big(G_{s_n\al}+G_\al\big).
\end{equation}
In particular, if $s_n\al=\al$ then
$H_n^{(s)}G_\al=-G_\al$ by Lemma \ref{th:9} c).
\end{list}
\end{proposition}
\begin{remark}
Note that the right-hand sides of \eqref{f1} and \eqref{f2} are well
defined by 
part c) of 
Corollary \ref{cor:9}.
Note furthermore that in \eqref{f1} the second formula does not reduce
to the first formula if we would assume $s_j\al=\al$. Indeed, part b)
of Lemma \ref{th:9} in case of the principal specialization
$\tau_i=st^{n-i}$ ($i\in [1,n]$) implies that
\begin{equation}\label{pcequal}
\wb\al_j/\wb\al_{j+1}=t\qquad
(j\in [1,n),\,\al\in\Lan\mbox{ with }s_j\al=\al).
\end{equation}
Hence the right-hand side of the first formula in \eqref{f1} equals
$(2t+1)G_\al$ when $s_j\al=\al$ ($j\in [1,n)$).
\end{remark}
\noindent
{\bf Proof of Proposition \ref{expansiontheorem}.}\quad\\
{\bf a)} The starting point is the formula (from \eqref{eq:24})
\begin{equation}
\big(H_j^{(t)}G_\al\big)\big(\,\wb\be\,\big)=
tG_\al\big(\,\wb\be\,\big)+
\frac{\wb\be_j-t\,\wb\be_{j+1}}
{\wb\be_j-\wb\be_{j+1}}\,
\big(G_\al\big(\bm{s_j}\wb\be\,\big))-
G_\al\big(\,\wb\be\,\big)\big)
\label{eq:26}
\end{equation}
for $\be\in\Lan$.  For $\be\in\La_{n,|\al|}\setminus\{\al\}$ we then
have by Lemma \ref{th:9} that
\begin{equation}\label{start}
\big(H_j^{(t)}G_\al\big)\big(\,\wb\be\,\big)=
\begin{cases}
0\qquad &\mbox{if $s_j\be\ne\al$},\sLP
\displaystyle\frac{\wb{\be}_j-t\,\wb{\be}_{j+1}}
{\wb{\be}_j-\wb{\be}_{j+1}}\qquad &\mbox{if $s_j\be=\al$}.
\end{cases}
\end{equation}
Now take $\be=\al$ and first
consider the case that $s_j\al=\al$.
Then \eqref{start} implies
that $H_j^{(t)}G_\al\in \FSP_{n, |\al|}$ vanishes
at the interpolation points
$\wb{\be}$ with $\be\in \La_{n,|\al|}\setminus\{\al\}$, while
we get from \eqref{eq:26} and Lemma \ref{th:9} that
\[
\big(H_j^{(t)}G_\al\big)(\wb{\al})=tG_\al(\wb{\al})=t.
\]
This proves the first formula of \eqref{f1}.
Now suppose that $s_j\al\not=\al$. 
Then \eqref{start} implies that $H_j^{(t)}G_\al\in\FSP_{n,|\al|}$ vanishes
at the interpolation points $\wb{\be}$ for
$\be\in \La_{n,|\al|}\setminus\{\al,s_j\al\}$.
Furthermore, by \eqref{eq:26} and Lemma \ref{th:9},
\begin{equation*}
\big(H_j^{(t)}G_\al\big)(\wb{s_j\al})=
\frac{\wb{\al}_{j+1}-t\wb{\al}_j}{\wb{\al}_{j+1}-\wb{\al}_j},
\qquad
\big(H_j^{(t)}G_\al\big)(\wb{\al})=\frac{(t-1)\wb{\al}_j}{\wb{\al}_j-
\wb{\al}_{j+1}}.
\end{equation*}
Hence the Laurent polynomial
\[
F:=H_j^{(t)}G_\al-
\frac{\wb{\al}_{j+1}-t\wb{\al}_j}{\wb{\al}_{j+1}-\wb{\al}_j}\,
G_{s_j\al}-
\frac{(t-1)\wb{\al}_j}{\wb{\al}_j-\wb{\al}_{j+1}}\,G_\al
\]
of degree at most $|\al|$
vanishes at all the interpolation points $\wb{\be}$ with
$\be\in \La_{n,|\al|}$, which
forces $F\equiv 0$. We conclude that
\[
H_j^{(t)}G_\al=
\frac{\wb{\al}_{j+1}-t\wb{\al}_j}{\wb{\al}_{j+1}-\wb{\al}_j}\,
G_{s_j\al}+
\frac{(t-1)\wb{\al}_j}{\wb{\al}_j-\wb{\al}_{j+1}}\,G_\al.
\]
Rewriting the right-hand side yields the second formula of \eqref{f1}.
\sLP
{\bf b)}
The proof proceeds similar to the proof of a),
now starting with the formula (from \eqref{eq:24})
\begin{equation}
\big(H_n^{(s)}G_\al\big)\big(\,\wb{\be}\,\big)=
-G_\al\big(\,\wb{\be}\,\big)+
\frac{\big(1-s\wb{\be}_n^{-1}\big)\big(1-s^{-1}\wb{\be}_n^{-1}\big)}
{1-\wb{\be}_n^{-2}}\,
\big(G_\al\big(\bm{s_n}\wb{\be}\,\big)-G_\al\big(\,\wb{\be}\,\big)\big)
\label{eq:27}
\end{equation}
for $\be\in\Lan$. By Lemma \ref{th:9} the formula 
reduces for $\be\in \La_{n,|\al|}\setminus\{\al\}$ to
\begin{equation}\label{start2}
\big(H_n^{(s)}G_\al\big)\big(\,\wb{\be}\,\big)=
\begin{cases}
0\qquad &\mbox{if $s_n\be\ne\al$,}\sLP
\displaystyle
\frac{\big(1-s\wb{\be}_n^{-1}\big)\big(1-s^{-1}\wb{\be}_n^{-1}\big)}
{1-\wb{\be}_n^{-2}}\qquad &\mbox{if $s_n\be=\al$.}
\end{cases}
\end{equation}
If $s_n\al=\al$ then it follows from part c) of Lemma \ref{th:9} in
case of principal specialization that $\wb\al_n=s$,
hence $(H_n^{(s)}G_\al)(\wb\al)=-G_\al(\wb\al)=-1$ by \eqref{eq:27}.
We conclude that
$H_n^{(s)}G_\al=-G_\al$ if $s_n\al=\al$, which agrees with \eqref{f2}.
If $s_n\al\not=\al$ then observe
that, by \eqref{eq:27} and Lemma \ref{th:9},
\begin{equation*}
\big(H_n^{(s)}G_\al\big)(\wb{s_n\al})=\frac{(1-s\wb{\al}_n)
(1-s^{-1}\wb{\al}_n)}{1-\wb{\al}_n^2},\qquad
\big(H_n^{(s)}G_\al\big)(\wb{\al})=\frac{2\wb{\al}_n^2-(s+s^{-1})
\wb{\al}_n}{1-\wb{\al}_n^2}.
\end{equation*}
Continuing the proof as in part a) readily leads to the formula
\eqref{f2}.\qed
\bPP
Recall from Lemma \ref{th:10} the representation $\pi_s$ of $\FSH(t,-1)$
on $\FSP_n$ and from \eqref{eq:28} the Hecke symmetrizer $C_+$.
\begin{lemma}
Let $s,t\in\CC^*$. The map
$\pi_s(C_+)$ restricts to a map
$\pi_s(C_+)\colon\FSP_{n,d}\to\FSP_{n,d}^{W_n}$.
\end{lemma}
\begin{proof}
The fact that $f^+:=\pi_s(C_+)f$ is $W_n$-invariant for 
$f\in\FSP_n$ follows from the following standard argument in
Cherednik-Macdonald theory: by \eqref{eq:29}
the Laurent polynomial $f^+$ satisfy
$(H_j^{(t)}-t)f^+=0$ for $j\in [1,n)$ and $(H_n^{(s)}+1)f^+=0$.
By the explicit forms \eqref{eq:24}, \eqref{eq:25}
of the Demazure--Lusztig operators, this is equivalent to $s_jf^+=f^+$
for $j\in [1,n)$ and $s_nf^+=f^+$.
The map $\pi_s(C_+)$ preserves $\FSP_{n,d}$ by Lemma ~\ref{th:11}.
\end{proof}
\begin{lemma}\label{steplem}
Let $s,t\in\CC$ such that $st^\rho\in\FST_n$.
Then 
\begin{equation}\label{s1}
\pi_s(C_+)G_\al=\textup{cst}_\al\, R_{\al^+}
\end{equation}
for all $\al\in\Lan$, with $\textup{cst}_\al:=\big(\pi_s(C_+)G_\al\big)(\wb{\al^+})$.
\end{lemma}
\Proof
To prove \eqref{s1} it suffices, in view of the previous lemma, to show that
\[
\big(\pi_s(C_+)G_\al\big)(\wb{\mu})=0
\]
for all $\mu\in \La_{n,|\al|}^+\setminus\{\al^+\}$. But by Proposition
\ref{expansiontheorem} we have
\[
\pi_s(C_+)G_\al=\sum_{\ga\in W_n\al}d_\ga^\al\, G_\ga
\]
for certain coefficients $d_\ga^\al\in\CC$. Each $G_\ga$ vanishes at 
$\wb{\mu}$ ($\mu\in\La_{n,|\al|}^+\setminus\{\al^+\}$) since
\[
\La_{n,|\al|}^+\setminus\{\al^+\}\subseteq \La_{n,|\ga|}\setminus\{\ga\}
\qquad \forall\, \ga\in W_n\al.
\]
This concludes the proof of the lemma.\qed
\bPP
We end this section by computing
$\textup{cst}_\la=(\pi_s(C_+)G_\la)(\wb\la)$ for $\la\in\La_n^+$. Fix
$\la\in\La_n^+$.  
Put $C_+^\la:=\sum_{u\in W_n^\la}T_u$ and
$C_{+,\la}:=\sum_{v\in W_{n,\la}}T_v$. Since
$\ell(uv)=\ell(u)+\ell(v)$ for $u\in W_n^\la$ and $v\in W_{n,\la}$ it
follows that
\begin{equation}\label{symmfactor}
C_+=C_{+}^\la C_{+,\la}
\end{equation}
in $\FSH_n(t,-1)$.
With $\chi$ the trivial one-dimensional representation of
$\FSH_n(t,-1)$ we see by Proposition \ref{expansiontheorem} that
\begin{equation}
\pi_s(C_{+,\la})G_\la=\chi(C_{+,\la})G_\la.
\label{eq:31}
\end{equation}
What remains is to compute $(\pi_s(C_+^\la)G_\la)(\wb\la)$.
Let $R=R_s\cup R_{\ell}$ be the decomposition of the root system
$R$ of type $C_n$ in short
and long roots (by convention, for $n=1$ we write $R_s=\emptyset$ and $R_\ell=R$).
We define two $W_n$-invariant functions
$R\rightarrow\CC$ 
by
\begin{equation*}
\kappa_\be:=
\begin{cases}
t\quad &\hbox{ if }\,\be\in R_s,\\
s\qquad &\hbox{ if }\, \be\in R_\ell,
\end{cases}
\qquad\qquad\qquad
\upsilon_\be:=
\begin{cases}
1\qquad &\hbox{ if }\, \be\in R_s,\\
-s^{-1}\qquad &\hbox{ if }\, \be\in R_\ell.
\end{cases}
\end{equation*}
Define for $\be\in R$ the rational function
\[
c_\be(x):=\frac{(x^{\be^\vee}-\kappa_\be)(x^{\be^\vee}+\upsilon_\be)}
{x^{2\be^\vee}-1},
\]
where $\be^\vee:=2\be/\|\be\|^2$ denotes the co-root of $\be$.
Then $H_j^{(t)}=t+c_{\be^j}(x)(s_j-1)$ for $j\in [1,n)$ and $H_n^{(s)}=-1+c_{\beta^n}(x)(s_n-1)$.
Furthermore, for $\al\in\La_n$ with $s_{j}\al\not=\al$ we have 
\begin{equation*}
\begin{split}
c_{e_{j+1}-e_j}(\wb\al)&=\frac{\wb{\al}_{j+1}-t\wb{\al}_j}
{\wb{\al}_{j+1}-\wb{\al}_j},\qquad\qquad\qquad j\in [1,n),\\
c_{-e_n}(\wb\al)&=\frac{(1-s\wb{\al}_n)(1-s^{-1}\wb{\al}_n)}
{1-\wb{\al}_n^2},
\end{split}
\end{equation*}
which are exactly the coefficients appearing in
\eqref{f1} and \eqref{f2}. 
Hence we have 
\begin{equation}\label{uniform}
\pi_s(T_j)G_\al=-G_\al+c_{-\be^j}(\wb\al)\bigl(G_{s_j\al}+G_\al\bigr)
\qquad (j\in [1,n]\;\al\in\Lan\mbox{ with }s_j\al\not=\al)
\end{equation}
by Proposition \ref{expansiontheorem}.

Recall the longest element $w_0=-\id_{\RR^n}$ in $W_n$.
Let $w_0^\la\in W_n^\la$ be the minimal coset representative
of the coset $w_0W_{n,\la}$.
%
\begin{theorem}
Let $s,t\in\CC$ such that $st^\rho\in\FST_n$.
Then
\[
\pi_s(C_+)G_\la=\textup{cst}_\la R_\la
\]
for $\la\in\La_n^+$, with 
\[
\textup{cst}_\la=
\chi(C_{+,\la})\prod_{\be\in R^+\cap (w_0^\la)^{-1}R^-}c_{-\be}(\wb\la).
\]
\end{theorem}
\Proof 
Fix $\la\in\La_n^+$ and write $\la^-:=w_0(\la)=w_0^{\la}(\la)$ for the
antidominant element in the orbit $W_n\la$.
By \eqref{symmfactor},  \eqref{eq:31}, Lemma \ref{steplem} and
Theorem \ref{symmversusnonsymm} we have
\[
\chi(C_{+,\la})\pi_s(C_+^\la)G_\la=\pi_s(C_+)G_\la=\textup{cst}_\la R_\la=
\textup{cst}_\la\sum_{\al\in W_n\la}G_\al,
\]
so it suffices to show that 
\[
\bigl(\pi_s(C_+^\la)G_\la\bigr)(\wb{\la_-})=
\prod_{\be\in R^+\cap (w_0^\la)^{-1}R^-}c_{-\be}(\wb\la).
\]

Using Proposition \ref{expansiontheorem}, the coefficient 
$e_{\la_-}^\la$ in the expansion
\[
\pi_s(T_{w_0^\la})G_\la=\sum_{\ga\in W_n\la}e_\ga^\la
G_\ga
\]
is the same as the coefficient of $G_{\la_-}$ in the expansion of
$\pi_s(C_+^\la)G_\la$ in 
nonsymmetric $BC_n$-type interpolation Macdonald polynomials. 
Hence it suffices to show that
\[
e_{\la_-}^\la=\prod_{\be\in R^+\cap (w_0^\la)^{-1}R^-}c_{-\be}(\wb\la).
\]
Choose a reduced expression $w_0^{\la}=s_{i_1}\cdots s_{i_r}$.
By the proof of Corollary \ref{cor:9}a)
the elements
$\la_k:=s_{i_k}\cdots s_{i_r}\la\in W_n\la$
($1\leq k\leq r+1$, with $\la_{r+1}:=\la$)
are pairwise distinct. In particular,
$\la_k=s_{i_k}\la_{k+1}\not=\la_{k+1}$ ($1\leq k\leq r$) and
$\la_1=\la_-$. It follows from Corollary \ref{cor:9}{\bf c)} that $c_{-\be^{i_k}}(\wb\la_{k+1})$ is well defined for $k\in [1,r]$,  
and part a) of Lemma \ref{th:9} implies that 
\[
c_{-\be^{i_k}}(\wb\la_{k+1})=c_{-s_{i_r}\cdots s_{i_{k+1}}\beta^{i_k}}(\wb\la).
\]
The coefficient $e_{\la_-}^\la$ can now be computed using Proposition
\ref{expansiontheorem} and \eqref{uniform},
\begin{equation*}
\begin{split}
e_{\la^-}^\la&=
c_{-\be^{i_1}}(\wb\la_2)\cdots
c_{-\be^{i_{r-1}}}(\wb\la_r)c_{-\be^{i_r}}(\wb\la_{r+1})\\
&=c_{-s_{i_r}\cdots s_{i_2}\be^{i_1}}(\wb\la)\cdots
c_{-s_{i_r}\be^{i_{r-1}}}(\wb\la)c_{-\be^{i_r}}(\wb\la)\\
&=\prod_{\be\in R^+\cap (w_0^\la)^{-1}R^-}c_{-\be}(\wb\la),
\end{split}
\end{equation*}
where the 
third equality follows from the well known description
of the set of positive roots mapped by $w_0^\la$ to negative roots
in terms of the reduced expression $w_0^\la=s_{i_1}\cdots
s_{i_r}$ (see Section \ref{sec:2}). This concludes the proof of the proposition.\qed
\appendix
\section{Appendix}
For $n=2$, $|\al|\leq 4$ and pseudo-random parameters $q,s,t$ from
$\{\tfrac1{100},\tfrac2{100},\ldots,\tfrac{99}{100}\}$ we have computed
the zeros of $\wb\be\mapsto G_\al(\wb\be;q,s,t)$ for $\be=(\be_1,\be_2)\in\La_2$ satisfying $\be_1,\be_2\in\{-10,\ldots,10\}$ using {\em Wolfram Mathematica} \cite{Wo}, and checked the results by computing the zeros once more for a second choice of
pseudo-random parameters.
We present in Figures 1--9 the outcome of these computations for $|\al|=4$.

The meaning of the colours of the dots in the pictures is as follows:
\begin{itemize}
\item
brown dot: $\al=(\al_1,\al_2)$.
\item
green dot: $(0,0)$.
\item
black dots: $\be\ne(0,0)$ for which $\overline\be$ is an interpolation point, i.e., $\beta\in\Lambda_{2,|\alpha|}$ and $\beta\not=\alpha$.
\item
red dots: $\be\in\{-10,\ldots,10\}^{\times 2}$ for which $\overline\be$ is an extra vanishing
point.
\end{itemize}
\begin{figure}[p]
\centering
\parbox{5cm}{
\includegraphics[width=5cm]{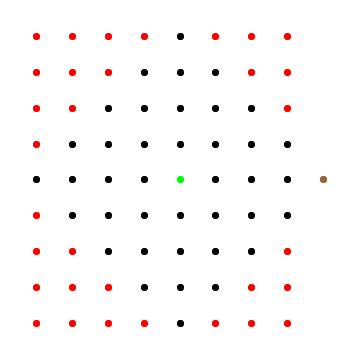}
\caption{$\al=(4,0)$}
\label{fig:1}}
\quad
\begin{minipage}{5cm}
\includegraphics[width=5cm]{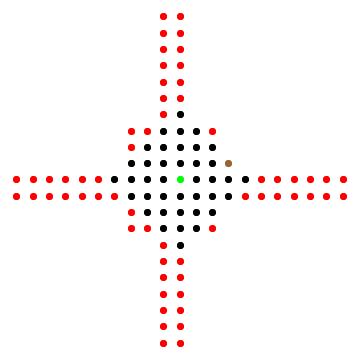}
\caption{$\al=(3,1)$}
\label{fig:2}
\end{minipage}
\quad
\begin{minipage}{5cm}
\includegraphics[width=5cm]{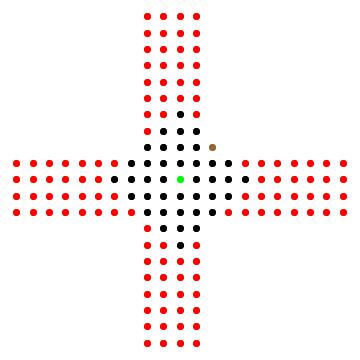}
\caption{$\al=(2,2)$}
\label{fig:3}
\end{minipage}
\end{figure}
\begin{figure}[p]
\centering
\parbox{5cm}{
\includegraphics[width=5cm]{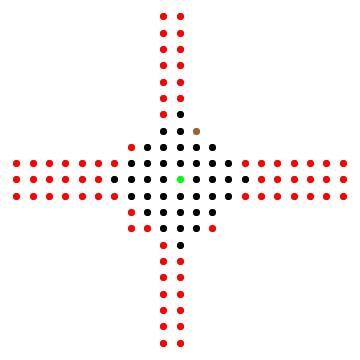}
\caption{$\al=(1,3)$}
\label{fig:4}}
\quad
\begin{minipage}{5cm}
\includegraphics[width=5cm]{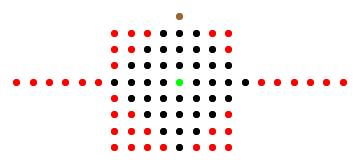}
\caption{$\al=(0,4)$}
\label{fig:5}
\end{minipage}
\quad
\begin{minipage}{5cm}
\includegraphics[width=5cm]{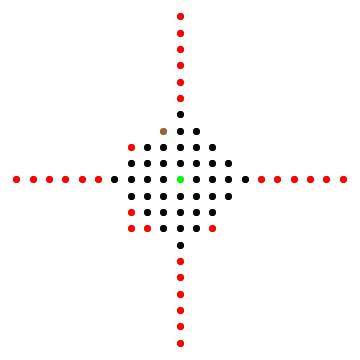}
\caption{$\al=(-1,3)$}
\label{fig:6}
\end{minipage}
\end{figure}
\begin{figure}[p]
\centering
\parbox{5cm}{
\includegraphics[width=5cm]{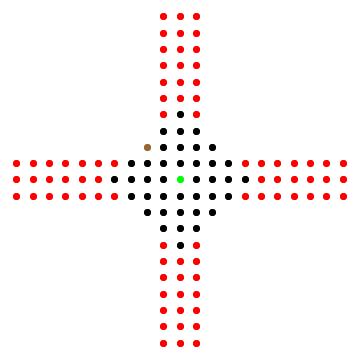}
\caption{$\al=(-2,2)$}
\label{fig:7}}
\quad
\begin{minipage}{5cm}
\includegraphics[width=5cm]{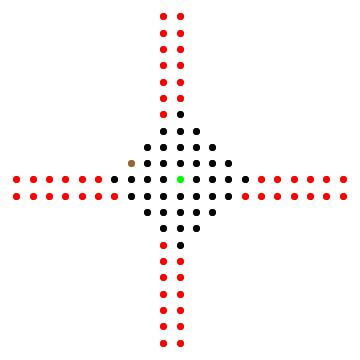}
\caption{$\al=(-3,1)$}
\label{fig:8}
\end{minipage}
\quad
\begin{minipage}{5cm}
\includegraphics[width=5cm]{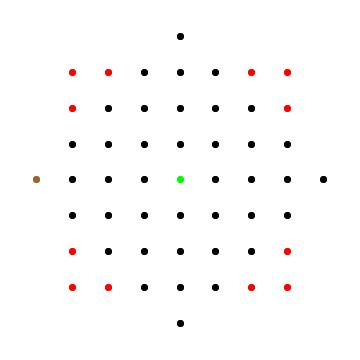}
\caption{$\al=(-4,0)$}
\label{fig:9}
\end{minipage}
\end{figure}

We have not included pictures for $(\al_1,\al_2)$ with $\al_2<0$ because the
following symmetry 
\[
G_{(\al_1,-\al_2)}(\overline{(\be_1,\be_2)};q,s,t)=0\quad\mbox{iff}\quad
G_{(\al_1,\al_2)}(\overline{(\be_1,-\be_2)};q,s,t)=0
\]
holds true in all the computed cases.

All our pictures, including the ones not displayed here, are in agreement with the following conjecture
about the zero set
\[
Z_\al:=\{\be\in\Lan\mid G_\al\big(\,\overline\be;q,s,t\big)=0\}
\]
for $q,s,t\in\CC^*$ with $q$ not a root of unity and $st^\rho\in\FST_n$.
\paragraph{Conjecture.}
Let $\al\in\Lan$, $\be\in W_n\al$.
Let $\overline{V_\be}$ consist of all $\mu\in\Lan$ such that, for $i=1,\ldots,n$,
$\mu_i\ge\be_i$, $\mu_i\le\be_i$ or $\mu_i\in\ZZ$ according to whether
$\be_i>0$, $\be_i<0$ or $\be_i=0$, respectively.
Let $V_\be^0$ consist of all $\mu\in\Lan$ such that, for $i=1,\ldots,n$,
$\mu_i>\be_i$, $\mu_i<\be_i$ or $\mu_i\in\ZZ\backslash\{0\}$
according to whether $\be_i>0$, $\be_i<0$ or $\be_i=0$, respectively.

Then there are sets $V_\be(\al)$ with
$V_\be^0\subseteq V_\be(\al)\subseteq\overline{V_\be}$ ($\be\in W_n\al$) and $V_\al(\al)=\overline{V_\al}$ such that $Z_\al$ is the complement in
$\Lan$ of $\union_{\be\in W_n\al} V_\be(\al)$.
\paragraph{Remark.}
If $\al\in\La_n^+$ (i.e., $\al$ is a partition) then for all $\be\in W_n\al$,
\[
\overline{V_\be}\cap\Lambda_n^+\subseteq
\{\mu\in\La_n^+\mid \mu\supseteq\al\},\quad\mbox{hence}\quad
V_\be(\al)\cap\Lambda_n^+\subseteq
V_\al(\al)\cap\Lambda_n^+=
\{\mu\in\La_n^+\mid \mu\supseteq\al\}.
\]
Thus the Conjecture implies that, for $\al$ a partition, $Z_\al\cap\La_n^+$
consists of all partitions $\mu$ which do not include the partition $\al$.
Compare with the case of Okounkov's $BC_n$-type interpolation Macdonald polynomials,
see~\eqref{extravansym}.
\bPP
Our pictures suggest possible characterizations of the sets $V_\be(\al)$
in the Conjecture. These seem to be quite similar
to the case of root system of type $A$
(see  \cite[Theorem 4.5]{Kn97})
if $n=2$, $\al_1>0$, $\al_2\ge0$,
or possibly for general $n$, $\al_1,\ldots,\al_{n-1}>0$, $\al_n\ge0$.
In contrast, in Figures 6 and 8, where $\al_1<0$, we see that
$\la\in V_\al(\be)$ is not
always given by one set of inequalities for $\la_1,\la_2$.

In \cite[\S 4]{Kn97} Knop introduced a new partial order on $\mathbb{Z}_{\geq 0}^n$ to describe the extra vanishing of the type $A_{n-1}$ nonsymmetric interpolation Macdonald polynomials. Knop's order relation between two elements $\alpha,\beta\in\mathbb{Z}_{\geq 0}^n$ can be described in terms of 
inequalities of the entries of the corresponding partitions $\alpha^+, \beta^+\in\Lambda_n^+$, with the strictness or non-strictness of the inequalities depending on the defining permutation $\pi=u_\alpha u_\beta^{-1}$ (here 
$u_\alpha, u_\beta\in S_n$ are the permutations of shortest lengths such that $u_\alpha(\alpha^+)=\alpha$ and 
$u_\beta(\beta^+)=\beta$). 
One may wonder whether the extra vanishing of the nonsymmetric 
$BC_n$-type interpolation Macdonald polynomials
$G_\al(x;q,s,t)$
can be formulated in terms of a hyperoctahedral version of Knop's partial order, with the strictness or non-strictness of the entries of the associated partitions 
$\alpha^+, \beta^+\in\Lambda_n^+$ now described in terms of $w_\alpha w_\beta^{-1}\in W_n$ for $\al,\be\in\Lambda_n$.

\end{document}